\documentclass[10pt]{amsart} 

\usepackage{amscd,amssymb} 
\usepackage{graphicx,url}
\usepackage{graphics}
\usepackage{rotating}
\usepackage{mathdots}
\usepackage{caption}
\usepackage{subcaption}
\usepackage{comment}
\usepackage{tikz}
\usetikzlibrary{arrows.meta}
\usepackage{amsmath}

\usepackage[backref,bookmarks=true,colorlinks=true,pdfstartview=FitV,linkcolor=blue, citecolor=red,urlcolor=green]{hyperref}

\newtheorem{thm}{Theorem}[section]
\newtheorem{lem}[thm]{Lemma}

\newtheorem{prop}[thm]{Proposition}

\newtheorem{question}[thm]{Question}

\theoremstyle{remark}
\newtheorem{rem}[thm]{Remark}

\theoremstyle{definition}
\newtheorem{defn}[thm]{Definition}

\usepackage{xcolor}

\newcommand\bR{{\mathbb{R}}}

\newcommand\Hom{{\rm Hom}}

\newcommand\ra{\rightarrow}

\newcommand\emp{\emptyset}

\newcommand\SL{{\mathsf{SL}}}
\newcommand\PSL{{\mathsf{PSL}}}
\newcommand\SO{{\mathsf{SO}}}

\newcommand{\Ad}{\mathrm{Ad}}

\newcommand{\Hit}{\mathrm{Hit}}

\newcommand{\vol}{\mathrm{vol}}

\newcommand{\PSO}{\mathsf{PSO}}
\newcommand{\Sp}{\mathsf{Sp}}
\newcommand{\PSp}{\mathsf{PSp}}
\newcommand{\Mod}{\operatorname{Mod}}

\long\def\deletefornow#1{}
\newcommand{\Max}{\operatorname{Max}}
\newcommand{\Span}{\operatorname{Span}}

\begin{document}
\title[Hitchin-Moduli spaces]{
The volumes of the Hitchin-Riemann moduli spaces are infinite} 

\author{Suhyoung Choi}
\author{Hongtaek Jung} 
\address{Department of Mathematics \\ KAIST \\Daejeon 305-701, South Korea}
\email{schoi@math.kaist.ac.kr}

\address{Department of Mathematical Sciences\\ Seoul National University\\ Seoul} 
\email{htjung@snu.ac.kr}

\date{\today}



\begin{abstract}
In this study, we prove that the actions of the mapping class groups on a large range of higher Teichm\"uller spaces with a rank of at least two possess infinite Atiyah-Bott-Goldman covolume. This result encompasses $\mathsf{G}$-Hitchin components of a higher rank split real form $\mathsf{G}$ and each component of the space of $\Sp_{2n}(\bR)$-maximal representations where $n \geq 2$. To achieve this outcome, we employ Goldman flows to find an infinite series of subsets of identical volume, the images of which in the quotient space are all mutually disjoint.

\end{abstract}
\maketitle

\section{Introduction} 

Let $S$ be a closed, orientable surface with a genus $g \ge 2$, and $\mathsf{G}$ a split real form of a simple Lie group of adjoint type. The $\mathsf{G}$-Hitchin-Teichm\"uller components, commonly referred to as Hitchin components, $\Hit_{\mathsf{G}}(S)$, of $S$ have been the subject of extensive research for the past thirty years. Presently, understanding regarding their topological types \cite{hitchin92}, symplectic structures \cite{Goldman84,SWZ2020}, and metric structures \cite{BCLS2015} is relatively comprehensive.

The mapping class group $\Mod(S)$ of $S$ acts properly on the Teichm\"uller space $\mathcal{T}(S)$, resulting in the Riemann moduli space $\mathcal{M}(S) = \mathcal{T}(S)/\Mod(S)$. In a similar fashion, Labourie \cite{Lab08} proved that the action of $\Mod(S)$ on Hitchin components $\Hit_{\mathsf{G}}(S)$ is also proper. From this, we can define the \emph{$\mathsf{G}$-Hitchin-Riemann moduli spaces} as $\mathcal{M}_{\mathsf{G}}(S) := \Hit_{\mathsf{G}}(S)/\Mod(S)$. Contrary to the Hitchin components themselves, the study of the Hitchin-Riemann moduli spaces appears to be in its initial stages. The principal aim of this paper is to reveal certain geometric properties of these Hitchin-Riemann moduli spaces.

The classical Riemann moduli spaces accommodate several supplementary structures. One such structure is the Weil-Petersson symplectic form. There is a direct equivalent of this in the form of the Atiyah-Bott-Goldman symplectic form $\omega_{ABG}$ on $\Hit_\mathsf{G}(S)$. Given that the $\Mod(S)$-action on $\Hit_\mathsf{G}(S)$ preserves $\omega_{ABG}$ \cite{Goldman84}, it allows us to consider the Atiyah-Bott-Goldman volume form denoted as $\vol:=\wedge^{\frac{1}{2}\dim \Hit_\mathsf{G}(S)}\omega_{ABG}$ on both $\Hit_\mathsf{G}(S)$ and $\mathcal{M}_\mathsf{G}(S)$. We shall now proceed to state our principal theorem on $\mathsf{G}$-Hitchin-Riemann moduli spaces for split real forms $\mathsf{G}$.

\begin{thm}[Main Theorem]\label{thm:main} 
Let $S$ be a closed orientable surface of genus $>1$, and let $\mathsf{G}$ be a split real form of a complex simple adjoint group. The total volume of the $\mathsf{G}$-Hitchin-Riemann moduli space $\mathcal{M}_{\mathsf{G}}(S)$ with respect to the volume form $\vol$ is infinite provided the real rank of $\mathsf{G}$ is at least two. 
\end{thm}

Now, $\mathsf{G}$-Hitchin components are well-known examples of \emph{higher Teichm\"uller spaces}. Another category of higher Teichm\"uller spaces originates from \emph{maximal representations}. Take a Lie group $\mathsf{H}$ of Hermitian type. A representation $\pi_1(S)\to \mathsf{H}$ that possesses the maximum feasible Toledo invariant can be viewed as a higher rank counterpart of Fuchsian representations. The collection of conjugacy classes of $\mathsf{H}$-maximal representations, denoted by 
\[ \operatorname{Max}_\mathsf{H}(S)=\{\rho\in \Hom(\pi_1(S),\mathsf{H})\,|\,\rho 
\text{ is maximal}\}/\mathsf{H}, \] 
comprises a union of connected components \cite{burger2010}.

Although $\Max_{\mathsf{H}}(S)$ may not be a smooth manifold, the smooth locus of $\operatorname{Max}_\mathsf{H}(S)$ supports the Atiyah-Bott-Goldman symplectic form. Moreover, $\Mod(S)$ acts properly on $\operatorname{Max}_\mathsf{H}(S)$, as shown in  \cite[Corollary~3.2]{wienhard06}. Therefore, we may consider the Atiyah-Bott-Goldman volume of the moduli space $\operatorname{Max}_\mathsf{H}(S)/\Mod(S)$. 

We will only focus on the case when $\mathsf{H}=\Sp_{2n}(\bR)$. According to \cite{guichard2010}, the moduli space $\operatorname{Max}_{\Sp_{2n}(\bR)}(S)/\Mod(S)$ consists of 6 connected components when $n>2$ and $2g+2$ components when $n=2$. Among them, 2 connected components are $\Sp_{2n}(\bR)$ lifts of the $\PSp_{2n}(\bR)$-Hitchin component. Thus, we already know that at least 2 connected components of $\operatorname{Max}_{\Sp_{2n}(\bR)}(S)/\Mod(S)$ have infinite volume. More generally, we exploit our method of proving Theorem~\ref{thm:main} to deduce the following: 
\begin{thm}\label{thm:maximal}
    Let $S$ be a closed orientable surface of genus $g\ge2$. Then the total volume of each connected component of  $\operatorname{Max}_{\mathsf{Sp}_{2n}(\bR)}(S)/\Mod(S)$ is infinite provided $n\ge 2$. 
\end{thm}

We can also define Hitchin-Riemann moduli spaces for compact surfaces with boundary. Let $S$ be  a compact orientable surface of negative Euler characteristic with oriented boundary components $\zeta_1, \cdots, \zeta_b$. We choose an ordered set $\mathbf{B}=(B_1, \cdots, B_b)$ of $\mathsf{G}$-conjugacy classes of elements that are conjugate into the positive Weyl chamber.  The relative Hitchin component $\Hit_\mathsf{G}(S, \mathbf{B})$ with respect to the boundary holonomy $\mathbf{B}$ is given by
\[
\Hit_\mathsf{G}(S, \mathbf{B})=\{[\rho] \in \Hit_\mathsf{G}(S)\,|\, \rho(\zeta_i)\in B_i,\,i=1,2,\cdots,b\}.
\]
For a compact surface $S$, we define its mapping class group $\Mod(S)$ to be the isotopy classes of self-homeomorphisms preserving an orientation and each boundary component set-wise. Because $\Mod(S)$ also acts on $\Hit_\mathsf{G}(S,\mathbf{B})$ properly, we can consider the moduli space $\mathcal{M}_\mathsf{G}(S, \mathbf{B})= \Hit_\mathsf{G}(S,\mathbf{B})/\Mod(S)$. The $\Mod(S)$-invariant symplectic structure $\omega_{GHJW}$ on $\Hit_\mathsf{G}(S, \mathbf{B})$ is due to Guruprasad-Huebschmann-Jeffrey-Weinsetin \cite{GHJW}. We also denote by $\vol$ the volume form on $\Hit_\mathsf{G}(S, \mathbf{B})$ and on $\mathcal{M}_\mathsf{G}(S,\mathbf{B})$ associated with $\omega_{GHJW}$.

The Weil-Petersson volume of the Riemann moduli space is finite and there is a recursive formula for computing its volume by Mirzakhani \cite{mirzakhani2007,mirzakhani20072}. Despite intimate relationship between $\mathcal{M}(S,\mathbf{B})$ and $\mathcal{M}_{\PSL_n(\bR)}(S,\mathbf{B})$,  Labourie and McShane \cite{LM09} suggested that the volume of $\mathcal{M}_{\PSL_n(\bR)}(S,\mathbf{B})$ is infinite for $n\ge 3$ when $S$ has a non-empty boundary. The technique for the proof of Theorem~\ref{thm:main} also confirms this claim:

\begin{thm}\label{thm:main2} 
Let $S$ be a compact orientable surface with $b>0$ boundary components of negative Euler characteristic. Let $\mathbf{B}=(B_1, \cdots, B_b)$ be a choice of boundary holonomy. The total volume of the relative $\PSL_n(\bR)$-Hitchin-Riemann moduli space $\mathcal{M}_{\PSL_n(\bR)}(S,\mathbf{B})$ with respect to the volume form $\vol$ is infinite provided $n\ge 3$. 
\end{thm}

We could not establish the relative version of Theorem~\ref{thm:main} for general split real forms, partly due to the lack of knowledge on the relative $\mathsf{G}$-Hitchin components $\Hit_\mathsf{G} (S, \mathbf{B})$. As a matter of fact, we do not even know whether $\Hit_\mathsf{G} (S, \mathbf{B})$ is non-empty for a given boundary holonomy $\mathbf{B}$, unless $\mathsf{G}=\PSL_n(\bR)$.

Our proofs for both Theorems~\ref{thm:main}, \ref{thm:main2} and \ref{thm:maximal} use a special Goldman flow, featured for instance in \cite{WZ18, FK, Goldman2013,goldman22}, along an essential simple closed curve $\alpha$ in $S$ together with the collar lemma obtained by Beyrer--Guichard--Labourie--Pozzetti--Wienhard \cite{beyrer2024} and Burger--Pozzetti \cite{burger2017}. The main idea is to find a small open set $N$ in the Hitchin component (or the maximal component) where the marked length spectrum have some kind of non-degeneracy. This non-degeneracy forces the set $\{\varphi \in \Mod(S)\,|\,\varphi (N) \cap N \ne \emptyset\}$ consist only of some Dehn twists along $\alpha$. Then we observe that if $N$ is flown sufficiently long time by the Goldman flow, no powers of the Dehn twist along $\alpha$ can bring it back to $N$. This allows us to construct an infinite family of pairwise disjoint open sets with the identical positive volume in the moduli space.

We hope to apply our result to other higher Teichm\"uller spaces. The most ambitious targets are the moduli spaces $\mathcal{M}_{\mathsf{G},\Theta}(S)$ of $\Theta$-positive representations \cite{bradlow2024, guichard2021} of a surface group $\pi_1(S)$ into a higher rank semi-simple Lie group $\mathsf{G}$ with a positive structure $\Theta$. Since they are universal models for all higher Teichm\"uller spaces, proving that $\mathcal{M}_{\mathsf{G},\Theta}(S)$ has infinite volume whenever $\operatorname{rank}\mathsf{G}>1$ will provide a unifying picture that the condition $\operatorname{rank}\mathsf{G}=1$ is really crucial for higher Teichm\"uller spaces to support the finite covolume action of the mapping class group. 

\subsection*{Outline of the paper} 
In Section~\ref{sec:hitchincomponent}, we recall some definitions of our main subjects, including Hitchin and maximal representations, Hitchin-Riemann moduli spaces and their symplectic structures. We first introduce the definition of $\mathsf{G}$-Hitchin representations for a surface-with-boundary. We also review some results on maximal representations. We also study the associated limit maps for Hitchin and maximal representations and show that one can choose a canonical limit map in Lemmas~\ref{lem:normalization} and ~\ref{lem:normalizationmax}. 

In Section~\ref{sec:goldmanflow}, we introduce algebraic bending deformations and Goldman flows. After developing algebraic bending, we explain how to understand the action of a Dehn twist in terms of an algebraic bending. We also discuss how to decompose a Hitchin and maximal component into Hitchin and maximal components of subsurfaces of smaller complexity. Finally, we show that algebraic bending deformations are proper in Proposition~\ref{prop:proper}. The properness of algebraic bending will play a key role. 

Section~\ref{mainsection} is devoted to proving our main results. We construct an open set that will never return to intersect itself under the $\Mod(S)$-action in Propositions~\ref{prop:length},~\ref{prop:lengthrel} and ~\ref{prop:lengthmax}.  Then in Lemma~\ref{lem:commute}, we show that, under mild assumptions, Goldman flows commute with the mapping class group action. The main step toward our main theorem is Lemma~\ref{lem:disj}, which shows that the open set constructed in Propositions~\ref{prop:length},~\ref{prop:lengthrel} and ~\ref{prop:lengthmax} will not come back by the mapping class group action. Given this lemma, the main results for the $\mathsf{G}$-Hitchin and of $\Sp_{2n}(\bR)$-maximal components follow easily. 

In Section~\ref{sec:question}, we append some remaining questions and future directions. Especially we discuss possible future candidates where our technique is applicable such as the moduli space of $\Theta$-positive representations. 

\subsection*{Acknowledgment} We thank Steven Bradlow, William Goldman, Oscar Garc\'{i}a-Prada, Johannes Huebschmann, Fran\c{c}ois Labourie, Gye-Seon Lee, Ilia Smilga and  Tengren Zhang for helpful conversations.

The first author was supported by the National Research Foundation of Korea (NRF) grant funded by the Korea government (MEST) (No. 2022R1A2C300316212). The second author was supported by the BK21 SNU Mathematical Sciences Division.

\section{Hitchin and maximal representations} \label{sec:hitchincomponent} 

In this section, we recall Hitchin and maximal representations and their symplectic structures. Unless otherwise stated, $S$ will denote a compact orientable surface of negative Euler characteristic.

\subsection{Lie groups}\label{sec:notation}
Let $\mathsf{G}$ be a connected semisimple Lie group. We know that the Lie algebra $\mathfrak{g}$ of $\mathsf{G}$ admits the Cartan decomposition $\mathfrak{g} = \mathfrak{k}\oplus \mathfrak{h}$ where $\mathfrak{k}$ is the maximal compact subalgebra. Let $\mathfrak{a}$ be the maximal abelian subspace of $\mathfrak{h}$. For a linear functional $\beta:\mathfrak{a}\to \bR$, we define
\[
\mathfrak{g}_\beta : = \{X\in \mathfrak{g}\,|\, [H,X] = \beta(H) X\,\text{ for all }H\in \mathfrak{a}\}
\]
and call $\beta$ a (restricted) root provided $\mathfrak{g}_\beta\ne\{0\}$. Let $\Sigma$ be the set of restricted roots of $\mathfrak{g}$.  The Lie algebra $\mathfrak{g}$ admits the restrict root space decomposition
\[
\mathfrak{g}=\mathfrak{g}_0 \oplus \bigoplus_{\beta\in \Sigma} \mathfrak{g}_\beta.
\]
We know that  $(\mathfrak{a}^* , \Sigma, \langle, \rangle)$ is a root system where  $\langle\cdot ,\cdot \rangle$ is the inner product on $\mathfrak{a}^*$ induced from the killing form. We introduce an order on the set of roots and let $\Sigma^+$ and $\Delta$ be the sets of  positive roots and simple roots of $\mathsf{G}$ respectively.

\begin{figure}[htb]
    \centering
    \includegraphics[width=\linewidth]{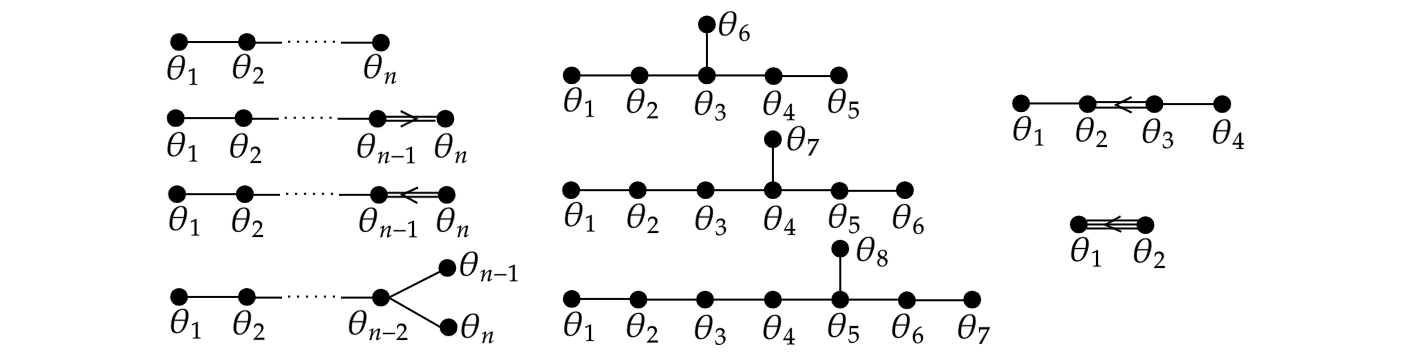}
    \caption{Enumeration of simple roots}
    \label{fig:simpleroots}
\end{figure}
We will enumerate the simple roots as in Figure~\ref{fig:simpleroots}. We will use this convention throughout this paper. A fundamental weight attached to a simple root $\theta_i\in \Delta$ is a linear functional $\omega_{\theta_i}:\mathfrak{a}\to \bR$ satisfying 
\[
2\frac{\langle \omega_{\theta_i}, \theta_j\rangle}{|\theta_j|^2} = \begin{cases}
    1 & i=j \\ 0 & i\ne j
\end{cases}.
\]
As $\Delta$ is a basis for $\mathfrak{a}^*$, we may write fundamental weights as linear combinations of simple roots. In particular, we will need the expression for $\omega_{\theta_1}+\omega_{-w_0\cdot \theta_1}$ where $w_0$ is the longest element in the Weyl group. One can find it by straightforward computation using the Cartan matrix. The result is summarized in Table~\ref{tab:fundamentalweights}. 
\begin{table}[htb]
    \centering
    \begin{tabular}{c|c}
    \hline
         Root system & Expression for $\omega_{\theta_1}+\omega_{-w_0\cdot \theta_1}$\\\hline
         $\mathsf{A}_n$& $\theta_1+\theta_2+\cdots +\theta_n$\\
         $\mathsf{B}_n$& $2\theta_1+2\theta_2 + \cdots + 2\theta_n$\\
         $\mathsf{C}_n$ & $2\theta_1+2\theta_2 + \cdots + 2\theta_{n-1}+\theta_n$\\
         $\mathsf{D}_n$ & $2\theta_1+2\theta_2+\cdots +2\theta_{n-2}+\theta_{n-1}+\theta_n$\\
         $\mathsf{E}_6$ & $2\theta_1 + 3\theta_2 +4 \theta_3 + 3\theta_4 + 2\theta_5 + 2\theta_6 $\\
         $\mathsf{E}_7$& $3\theta_1 + 4 \theta_2 +5\theta_3 +6 \theta_4+4\theta_5 +2 \theta_6 +3\theta_7$\\
         $\mathsf{E}_8$& $4\theta_1+6\theta_2+8\theta_3+10\theta_4+12\theta_5+8\theta_6+4\theta_7+6\theta_8$\\
         $\mathsf{F}_4$& $4\theta_1+6\theta_2+8\theta_3+4\theta_4$\\
         $\mathsf{G}_2$& $4\theta_1+6\theta_2$\\
         \hline
    \end{tabular}
    \caption{}
    \label{tab:fundamentalweights}
\end{table}

For a given split real form $\mathsf{G}$ of a complex adjoint simple Lie group, there are particularly nice representations of $\PSL_2(\bR)$ into $\mathsf{G}$. Let $\mathfrak{g}$  be the Lie algebra of $\mathsf{G}$. A \emph{three dimensional subalgebra} (TDS, in short) is a Lie subalgebra generated by nilpotent elements $E,F$ and a semisimple element $H$ in $\mathfrak{g}$ satisfying the relations $[H,E] =2E$, $[H,F]=-2F$ and $[E,F]=H$. For each positive root $\beta\in \Sigma^+$, we can find an associated $\mathfrak{sl}_2$-triple $E_\beta\in \mathfrak{g}_\beta$, $F_\beta\in \mathfrak{g}_{-\beta}$, and $H_\beta\in \mathfrak{a}$.  Also if we choose $E$, $F$ and $H$ as regular elements, the resulting TDS is called principal. A principal TDS induces a Lie group representation of $\PSL_2(\mathbb{C})$ into the complexification of $\mathsf{G}$. If we take the real points of it, we get the preferred representation $\tau_{\mathsf{G}}:\PSL_2(\bR)\to \mathsf{G}$.

Throughout this paper we will focus on Hitchin and maximal representations. When we discuss Hitchin representations, $\mathsf{G}$ will denote a split real form of a complex adjoint simple Lie group. When we are dealing with maximal representations $\mathsf{G}$ will denote the real symplectic group $\Sp_{2n}(\bR)$.

\subsection{Hitchin components}

The Hitchin components are originally defined for closed surfaces with genus at least 2. However, one can naturally generalize the definition to surfaces with boundary.

A $\mathsf{G}$-\emph{Hitchin representation} for a compact orientable surface $S$ possibly with boundary is a representation $\rho: \pi_1(S)\to \mathsf{G}$ such that there is a continuous family $\rho_t$ of representations with properties
    \begin{itemize}
        \item $\rho_0 = \rho$ and $\rho_1$ is a Fuchsian representation
        \item For each $t$ and each boundary component $\zeta$, we request that each $\rho_t(\zeta)$ is a loxodromic element, i.e., conjugate to an element in a positive Weyl chamber of $\mathsf{G}$. 
    \end{itemize}
Here, a \emph{Fuchsian representation} is any representation of the form $\tau_\mathsf{G} \circ \rho$ where $\rho$ is the holonomy of a hyperbolic metric on $S$ with totally geodesic boundary and $\tau_\mathsf{G}:\PSL_2(\bR) \to \mathsf{G}$ is the preferred representation defined in Section~\ref{sec:notation}. 

For $\mathsf{G}=\PSL_n(\bR)$, this definition coincides with the one in Labourie--McShane \cite{LM09}.

The set of Hitchin representations is a union of certain connected components of $\Hom(\pi_1(S),\mathsf{G})$. Let $\Hom_{\mathrm{H}}(\pi_1(S), \mathsf{G})$ be one of connected components of Hitchin representations and let $\Hit_\mathsf{G}(S):=\Hom_{\mathrm{H}}(\pi_1(S),\mathsf{G})/\mathsf{G}$. By convention, denote  $\Hit_{\PSL_n(\bR)}(S)$ simply by $\Hit_n(S)$.

The following lemma is essentially due to Zhang \cite[Proposition 3.5]{zhang2015} based on Bonahon--Dreyer \cite{bonahon2014}. 

\begin{lem}[\cite{zhang2015}]\label{boundaryfreedom}
    Let $\alpha_1, \cdots, \alpha_k$ be  pairwise disjoint non-isotopic essential simple closed curves in $S$. Let $C_1, \cdots, C_k$ be conjugacy classes of diagonal matrices with distinct positive eigenvalues. Let $\mathbf{B}=(B_1, \cdots, B_b)$ be any choice of boundary holonomy. Then there is $[\rho]\in \Hit_n(S, \mathbf{B})$ such that $\rho (\alpha_i ) \in C_i$ for $i=1,2,\cdots, k$.
\end{lem}

\begin{rem}
    To our knowledge, Lemma~\ref{boundaryfreedom} is known only for $\PSL_n(\bR)$-Hitchin representations. 
\end{rem}

\subsection{Maximal representations}
Throughout this paper, we will only deal with $\Sp_{2n}(\bR)$-maximal representations. Note however that maximal representations can be defined for general Hermitian type Lie groups. 

Given a representation $\rho\in \Hom(\pi_1(S), \Sp_{2n}(\bR))$, we associate so-called the \emph{Toledo invariant} $T(\rho)\in \mathbb{Z}$. In \cite{burger2010}, it is showed that this number takes integral values and falls into the range $-n|\chi(S)|\le T(\rho)\le n|\chi(S)|$. Call a representation $\rho$ \emph{maximal} if $T(\rho) =  n|\chi(S)|$ and let
\begin{align*}
\Hom_{\mathrm{M}}(\pi_1(S),\Sp_{2n}(\bR)) &:=\{\rho \in \Hom(\pi_1(S),\Sp_{2n}(\bR))\,|\,\rho\text{ is maximal}\},\\ 
\Max_{\Sp_{2n}(\bR)}(S)&:=\Hom_{\mathrm{M}}(\pi_1(S),\Sp_{2n}(\bR)) /\Sp_{2n}(\bR). 
\end{align*}
Since we focus only on $\Sp_{2n}(\bR)$-maximal representations, write simply $\Max_{2n}(S) = \Max_{\Sp_{2n}(\bR)}(S)$. 

$\Max_{2n}(S)$ consists of many connected components. If $S$ is a closed surface with genus $g$, $\operatorname{Max}_{2n}(S)$ has $3\cdot 2^{2g}$ components when $n>2$ \cite{garcia2013} and $3\cdot 2^{2g}+2g-4$ components when $n=2$ \cite{garcia2004}. Moreover, Guichard-Wienhard \cite{guichard2010} showed that each connected component of $\Hom_{\mathrm{M}}(\pi_1(S),\Sp_{2n}(\bR))$ contains explicit representations that we will describe below. 

Consider the standard irreducible embedding $\tau_{2n}:\SL_2(\bR)\to \SL_{2n}(\bR)$. By adopting the identification  \cite[Section 3.2.1]{guichard2010},  we may understand $\tau_{2n}$ as the embedding $\tau_{2n}:\SL_2(\bR)\to \Sp_{2n}(\bR)$. On the other hand, we can decompose $\bR^{2n}=\bR^n \otimes \bR^2$ where $\bR^n$ is given the Euclidean inner product $q$ and $\bR^2$ is equipped with the standard symplectic form $\omega$. Then the bilinear form $q\otimes \omega$ given by 
\[
(x\otimes z, y\otimes w)\mapsto q(x,y)\omega(z,w)
\]
is a symplectic form on $\bR^n \otimes \bR^2$. Define the isomorphism $\bR^n\otimes \bR^2\to \bR^{2n}$ by 
\[
e_i\otimes e_1\mapsto e_i,\qquad e_i\otimes e_2 \mapsto e_{i+n}. 
\]
This map gives rise to an isomorphism $\delta_{2n}: \Sp(\bR^n\otimes \bR^2)\to \Sp_{2n}(\bR)$.  Given a lifted  Fuchsian representation $\rho:\pi_1(S) \to \SL_2(\bR)$, $\delta_{2n}\circ(I_{n}\otimes \rho)$ and $\tau_{2n}\circ \rho$ are maximal representations, called \emph{standard diagonal} and \emph{standard irreducible representations} respectively. Here $I_n$ denotes the $n\times n$ identity matrix. Note also that the centralizer of a standard diagonal representation is the maximal compact subgroup $\mathsf{O}_n$. This allows us to twist $\delta_{2n} \circ(I_n\otimes \rho)$ by a representation $\kappa: \pi_1(S)\to \mathsf{O}_n$, to obtain another maximal representation
\[
\delta_{2n}\circ(\kappa \otimes  \rho).
\]
Such representations are called \emph{twisted diagonal representations}. Both twisted diagonal and irreducible maximal representations are referred to \emph{standard maximal representations}.

When $n=2$, we have to consider another type of maximal representations. Take a separating simple closed curve $\alpha$ in $S$ cutting $S$ into subsurfaces $S_1$ and $S_2$ and let $\rho_2:\pi_1(S_2)\to \SL_2(\bR)$ be a lifted Fuchsian representation. Up to conjugation, we may write $\rho_2(\alpha) = \operatorname{diag}(e^l,  e^{-l})$ for some $l>0$. We continuously deform $\rho_2$ to another representation $\rho_2':\pi_1(S_2)\to \SL_2(\bR)$ so that $\rho_2'(\alpha) = \operatorname{diag}(e^{3l},e^{-3l})$. Let us take another lifted Fuchsian representation $\rho_1:\pi_1(S_1)\to \SL_2(\bR)$ such that $\rho_1(\alpha)=\operatorname{diag}(e^l, e^{-l})$. Since 
\[
\tau_{2n}\circ\rho_1(\alpha) = \operatorname{diag}(e^{3l},e^l,e^{-3l},e^{-l}),
\]
one can show that the irreducible representation $\tau_{2n}\circ\rho_1$ and the representation $\gamma \mapsto \psi(\rho_2(\gamma), \rho_2'(\gamma))$
can be glued along $\alpha$, where
\[
\psi\left(\begin{pmatrix}
    a & b \\ c & d
\end{pmatrix},\begin{pmatrix}
    p &q \\ r & s
\end{pmatrix}\right)=\begin{pmatrix}
    a & 0 & b & 0 \\
    0 & p & 0 & q \\
    c & 0 & d & 0 \\
    0 & r & 0 &s
\end{pmatrix}.
\] Such amalgamated representations are called \emph{standard hybrid representations}.

\begin{thm}[Theorems 11 and 14 of \cite{guichard2010}]\label{thm:ComponentContainsFuchsian}
    Each connected component of $\Max_{2n}(S)$ contains a standard maximal representation or a standard hybrid representation. 
\end{thm}

\begin{lem}\label{lem:simpleeigenvalues}
   Let $\rho_0\in \Hom_{\mathrm{M}}(\pi_1(S),\Sp_{2n}(\bR))$ and let $U$ be any neighborhood of $\rho_0$. Let $\gamma$ be a non-separating essential simple closed curve in $S$ realized as an element of $\pi_1(S)$. Then there is $\rho\in U$ such that eigenvalues of $\rho(\gamma)$ are simple. If all eigenvalues of  $\rho_0(\gamma)$ are real, we may find $\rho\in U$ such that $\rho(\gamma)$ is loxodromic.
\end{lem}
\begin{proof} 
   We use Fenchel-Nielsen type parameters introduced in \cite{strubel2015}.  
   
    Decompose $S$ into a collection of pair-of-pants and once-punctured tori such that $\gamma$ is contained in one of once-punctured tori $S'$ as an essential simple closed curve. Let $\mathcal{B}\subset \mathsf{GL}_n(\bR)$ be the set of matrices whose absolute values of the (complex) eigenvalues are bigger than 1. Let $(Y_0,H_0,Z_0)\in\mathcal{B}\times \overline{\mathcal{B}}\times \mathcal{B}$  be coordinate matrices corresponding to $\rho_0|\pi_1(S')$. The matrices $Y_0$, $H_0$ and $Z_0$ are subject to the following condition:
    \begin{equation}\label{eq:coordinate}
       Z_0Y_0 ^{\mathsf{T}} Z_0 ^{-1} (H_0 ^{\mathsf{T}})^{-1} Y_0 \text{ is symmetric and positive definite.}
     \end{equation} 
    In this coordinate, $\rho_0(\gamma)$ is of the form
    \begin{equation}\label{eq:standard}
    \begin{pmatrix}
        Y_0 & O \\ * & (Y_0 ^{-1})^{\mathsf{T}}
    \end{pmatrix}.
    \end{equation}
    Conversely, one can find an open neighborhood $U_1\times U_2\times U_3\subset \mathcal{B}\times \overline{\mathcal{B}}\times \mathcal{B}$ of $(Y_0,H_0,Z_0)$ such that for any triple of matrices $(Y,H,Z)\in U_1\times U_2\times U_3$ satisfying the above conditions (\ref{eq:coordinate}), there is $\rho\in U$ such that $\rho(\gamma)$ is of the form   
    \[
    \begin{pmatrix}
        Y & O \\ * & (Y ^{-1})^{\mathsf{T}}
    \end{pmatrix}.
    \]

    Let $S_0:=Z_0Y_0 ^{\mathsf{T}} Z_0 ^{-1} (H_0 ^{\mathsf{T}})^{-1} Y_0$. This is a symmetric and positive definite matrix. Consider the map $f:\mathsf{GL}_n(\bR)\to\mathsf{GL}_n(\bR)$ given by  
    \[
    f:Y\mapsto (Z_0^{\mathsf{T}})^{-1} Y Z_0 ^{\mathsf{T}}(S_0 ^{-1})^{\mathsf{T}} Y^{\mathsf{T}}.
    \]
    Then, $f$ is a continuous map with $f(Y_0) = H_0$. Choose a continuous function $c:U_1\to [1,\infty)$ such that $c(Y_0)=1$ and that the absolute values of eigenvalues of $c(Y) f(Y)$ are bigger than 1. Define the continuous map $h:U_1 \to \overline{\mathcal{B}}$ by $Y\mapsto c(Y) f(Y)$. Note that $h(Y_0) = H_0$. 

    One can find $Y \in h^{-1}(U_2)\cap U_1$ whose eigenvalues are simple and strictly bigger than 1. Moreover if all eigenvalues of $Y_0$ are real, we can find a loxodromic $Y\in h^{-1}(U_2)\cap U_1$. Since $(Y,h(Y),Z_0)$ satisfies the condition (\ref{eq:coordinate}), there is a maximal representation $\rho\in U$ such that the  eigenvalues of $\rho(\gamma)$ are all simple. We also know that if eigenvalues of $\rho_0(\gamma)$ are all real, $\rho(\gamma)$ is loxodromic.
    \end{proof}

Note that non-Hitchin components of $\Max_{2n}(S)$ are singular because twisted diagonal representations have centralizer $\mathsf{O}_n$. To avoid singularities, we define 
\begin{align*}
    \Hom_{\mathrm{M}}^+(\pi_1(S),\Sp_{2n}(\bR))&:=\{\rho\in \Hom_{\mathrm{M}}(\pi_1(S), \Sp_{2n}(\bR))\,|\, Z(\rho) =\{1\}\}\\
    \Max_{2n} ^+ (S)&:=\Hom_{\mathrm{M}} ^+ (\pi_1(S),\Sp_{2n}(\bR))/\Sp_{2n}(\bR),
\end{align*}
where $Z(\rho)$ is the centralizer of $\rho$. It is a standard fact that $\Max_{2n}^+(S)$ is a smooth manifold.

\begin{lem}\label{lem:generic}
    $\Max_{2n} ^+ (S)$ is open and dense in $\Max_{2n}(S)$. 
\end{lem}
\begin{proof}
     Because  $\Hom_{\mathrm{M}}^+ (\pi_1(S), \Sp_{2n}(\bR))$ is Zariski open, we know that $\Max_{2n} ^+ (S)$ is open in $\Max_{2n}(S)$ and is dense in each connected component of $\Max_{2n}(S)$ that has non-trivial intersection with $\Max_{2n} ^+ (S)$. Hence, it suffices to show that  $\Hom_{\mathrm{M}}^+ (\pi_1(S), \Sp_{2n}(\bR))$ and each connected component of $\Hom_{\mathrm{M}}(\pi_1(S),\Sp_{2n}(\bR))$ have non-empty intersection. 
     
     Combining Theorem~\ref{thm:ComponentContainsFuchsian} and \cite[Theorem 7]{burger2010}, we know that each connected component of $\Hom_{\mathrm{M}}(\pi_1(S),\Sp_{2n}(\bR))$ contains a Zariski dense representation. Since the centralizer of a Zariski dense representation is trivial, $\Hom_{\mathrm{M}}^+ (\pi_1(S), \Sp_{2n}(\bR))$ intersects each connected component of $\Hom_{\mathrm{M}}(\pi_1(S),\Sp_{2n}(\bR))$ as we wanted.
\end{proof}

\subsection{Limit maps and positivity} 

Both Hitchin representations and maximal representations enjoy positivity. We briefly recall the notion of positive representation and drive some lemmas that we will use in the next section. For general theory of positive representations, the reader should consult \cite{guichardpositivity,guichard2021}.

Recall that the sets of restricted roots, positive roots and simple roots of $\mathsf{G}$ are respectively denoted by $\Sigma$, $\Sigma^+$ and $\Delta$.

We choose a subset $\Theta$ of $\Delta$ and denote $\Sigma_\Theta ^+ = \Sigma^+ \setminus \Span (\Delta \setminus \Theta)$. For such a given $\Theta$ and an element $\theta\in \Theta$, we define 
\begin{align*}
\mathfrak{u}_\Theta &:= \sum_{\beta\in \Sigma_\Theta ^+ }\mathfrak{g}_\beta,\\
\mathfrak{u}_{\Theta}^{\operatorname{op}}&:=\sum_{\beta\in \Sigma_{\Theta}^+}\mathfrak{g}_{-\beta},\\
\mathfrak{t}_\Theta &:= \bigcap_{\beta \in\Delta \setminus \Theta} \ker \beta,\\ 
\mathfrak{u}_\theta &:= \sum_{\substack{\beta\in \Sigma\\ \beta|\mathfrak{t}_\Theta = \theta}} \mathfrak{g}_\beta.
\end{align*}
Let $\mathsf{P}_\Theta$ and $\mathsf{P}^{\operatorname{op}}_\Theta$ be the subgroups of $\mathsf{G}$ normalizing $\mathfrak{u}_\Theta$ and $\mathfrak{u}_\Theta ^{\operatorname{op}}$ respectively. They are called the standard parabolic subgroup and the standard opposite parabolic subgroup associated with $\Theta$. We take the identity component $\mathsf{L}_\Theta ^0$ of the Levi factor of $\mathsf{P}_\Theta \cap \mathsf{P}_{\Theta} ^{\operatorname{op}}$.

We say that $\mathsf{G}$ admits the $\Theta$-\emph{positive structure} if for each $\theta\in \Theta$, there is an acute convex open cone $c_\theta$ in $\mathfrak{u}_\theta$ invariant under the $\mathsf{L}_\Theta ^0$-action. 

Positive structures of $\mathsf{G}$ are classified \cite{guichardpositivity}; The $\Theta$-positive structures that are relevant to our case are the followings.
\begin{itemize}
    \item[(a)] $\mathsf{G}$ is a split real form of a complex simple adjoint group and $\Theta = \Delta$;
    \item[(b)] $\mathsf{G} = \Sp_{2n}(\bR)$ and $\Theta = \{\theta_n \}$ where $\theta_n$ is the longest root. 
\end{itemize}
We will see that the case (a) is associated with $\mathsf{G}$-Hitchin representations, whereas the case (b) is related to $\Sp_{2n}(\bR)$-maximal representations. 

Suppose that we are given a $\Theta$-positive structure of $\mathsf{G}$. We can define the unipotent $\Theta$-positive semigroup $\mathsf{U}_\Theta ^{>0}$ to be the semigroup generated by the exponentials of the cones $\exp( c_\theta)$, $\theta\in \Theta$. 
\begin{thm}[Theorem 10.1 of \cite{guichardpositivity}]\label{thm:parameter}
Fix a reduced word expression $s_{i_1}\cdots s_{i_\ell}$ of the longest element of the $\Theta$-Weyl group. The map
\[
c_{\theta_{i_1}}\times \cdots \times c_{\theta_{i_\ell}} \to \mathsf{U}_\Theta ^{>0}
\]
given by 
\[
(v_1, \cdots, v_\ell) \mapsto \exp (v_1) \cdots \exp (v_\ell)
\]
is a diffeomorphism. 
\end{thm}
When $\Theta= \Delta$, we know that the $\Theta$-Weyl group is the same as the usual Weyl group of $\mathsf{G}$ and each cone $c_\theta$ is 1-dimensional and is generated by the element  $E_\theta$. Therefore, the theorem shows that $\mathsf{U}_\Theta ^{>0}$ can be parameterized by the positive $\ell$-tuples of positive real numbers
\[
(t_1, \cdots, t_\ell) \mapsto \exp (t_1 E_{\theta_{i_1}}) \cdots \exp (t_\ell E_{\theta_{i_\ell}}).
\]

Let $\mathcal{F}_\Theta= \mathsf{G}/\mathsf{P}_\Theta$ be the flag manifold associated with a $\Theta$-positive structure of $\mathsf{G}$. $\mathcal{F}_\Theta$ can be identified with the $\mathsf{G}$-orbit of the Lie algebra $\mathfrak{p}_\Theta$ of the standard parabolic subgroup. Observe that $\mathfrak{p}_\Theta ^{\operatorname{op}}=w_\Theta \mathfrak{p}_\Theta$ where $w_\Theta$ is the longest element of the $\Theta$-Weyl group. 

We say that a triple $(x,y,z)$ of distinct elements in $\mathcal{F}_\Theta$ is \emph{positive} if there is $g\in \mathsf{G}$ such that $(g\cdot x ,g\cdot y, g\cdot z)= (\mathfrak{p}_\Theta, u\cdot \mathfrak{p}^{\operatorname{op}}_\Theta ,\mathfrak{p}^{\operatorname{op}}_\Theta)$ for some $u\in \mathsf{U}_\Theta ^{>0}$. An $n$-tuple $(x_1, x_2, \cdots, x_n)$ of distinct elements of $\mathcal{F}_\Theta$ is positive if every triple  $(x_i, x_{i+1}, x_{i+2})$ (cyclically indexed) is positive. Finally, we say that a representation $\rho:\pi_1(S) \to \mathsf{G}$ is $\Theta$-\emph{positive} if there is a continuous $\rho$-equivariant limit map $\xi_\rho:\partial_\infty \pi_1(S) \to \mathcal{F}_\Theta$  sending positive tuples to positive tuples, where $\partial_\infty \pi_1(S)$ is given a cyclic order induced from an orientation of $S$. The following theorem is now well-known; see \cite{fock,burger2010,guichard2021}.

\begin{thm} Let $S$ be a compact orientable surface of negative Euler characteristic. 
\begin{itemize}
    \item[(i)] Let $\mathsf{G}$ be a split real form of a complex simple adjoint group. Then 
    $\mathsf{G}$-Hitchin representations of $\pi_1(S)$ are $\Theta$-positive for $\Theta =\Delta$. 
    \item[(ii)] $\Sp_{2n}(\bR)$-maximal representations of $\pi_1(S)$ are $\Theta$-positive for $\Theta=\{\theta_n\}$. 
\end{itemize}
\end{thm}

Consider a split real form $\mathsf{G}$ with $\Theta = \Delta$. Enumerate $\Delta= \{\theta_1, \cdots, \theta_r\}$ as in Figure~\ref{fig:simpleroots} in Section~\ref{sec:notation}. Note that a reduced word expression of $w_0$ contains each root reflection $s_{\theta_i}$ at least once. In view of Theorem~\ref{thm:parameter}, each $u\in \mathsf{U}_\Theta ^{>0}$ can be written  as
\[
u= \cdots \exp (t_1 E_{\theta_1})\cdots \exp (t_2 E_{\theta_2}) \cdots \exp (t_3 E_{\theta_3})\cdots \exp (t_r E_{\theta_r})\cdots.
\]
Thus, we can find a smooth map 
\[
f_\Theta:  \mathsf{U}_\Theta ^{>0}\to (\bR _ {>0})^r
\]
defined by $f_\Theta(u) = (t_1, \cdots, t_r)$. Choose $b_0 := (1,2,\cdots, r)\in (\bR_{>0})^r$ and let
\[
U_\mathsf{G}:=\{(\mathfrak{p}_\Theta, u\cdot \mathfrak{p}_\Theta^{\operatorname{op}}, \mathfrak{p}_\Theta^{\operatorname{op}})\,|\, u\in f_\Theta^{-1}(b_0)\}.
\]

The following lemma will be used in the next section. 

\begin{lem}\label{lem:normalization} Let $S$ be a compact orientable surface of negative Euler characteristic.  Give an orientation to $\partial_{\infty}\pi_1(S)$.  Let $\mathsf{G}$ be  a split real form and $\Theta= \Delta$. Let $\rho:\pi_1(S)\to \mathsf{G}$ be a $\Theta$-positive representation with the limit map  $\xi_\rho:\partial_\infty \pi_1(S)\to \mathcal{F}_\Theta$.  For a positively ordered triple $(x,z,y)\in \partial_\infty \pi_1(S)$, there is an element $g\in \mathsf{G}$ such that 
\[(g\cdot \xi_\rho (x), g\cdot \xi_\rho(z), g\cdot \xi_\rho(y))\in U_\mathsf{G}
\] 
Moreover, the set
\[
V:=\{g\in \mathsf{L}_\Theta\,|\,g\cdot U_\mathsf{G} \subset U_\mathsf{G}\}
\]
is compact and $[V,\mathsf{L}_\Theta ^0]=1$. 
\end{lem}
\begin{proof}
    By definition, one can find $g\in \mathsf{G}$ such that 
    \[
    (g\cdot \xi_\rho (x), g\cdot \xi_\rho(z), g\cdot \xi_\rho(y))=(\mathfrak{p}_\Theta, u'\cdot \mathfrak{p}_\Theta^{\operatorname{op}}, \mathfrak{p}_\Theta^{\operatorname{op}})
    \]
    for some $u'\in \mathsf{U}_\Theta ^{>0}$. We also observe that the stabilizer of the pair $(\mathfrak{p}_\Theta, \mathfrak{p}_{\Theta}^{\operatorname{op}})$ is $\mathsf{L}_\Theta:=\mathsf{P}_{\Theta}\cap \mathsf{P}_{\Theta}^{\operatorname{op}}$. Therefore, one has to find $g'\in \mathsf{L}_\Theta$ such that $f_\Theta(g'\cdot u') = b_0$. The existence of such an element $g'\in \mathsf{L}_\Theta$ is already shown by Guichard-Wienhard \cite[Remark~5.4]{guichardpositivity}. 
    
    By \cite[Proposition~10.3]{guichardpositivity}, we know that
    \[
    V=\{g\in \mathsf{L}_\Theta \,|\, f_\Theta(g ug^{-1})=b_0\}.
    \]
    Therefore it suffices to show that $\{g\in \mathsf{L}_\Theta \,|\, f_\Theta(g ug^{-1})=b_0\}$ is compact. In fact, since $\mathsf{L}_\Theta$ has finitely many connected components, we only need to show that  
    \[
    \{g\in \mathsf{L}_\Theta ^0 \,|\, f_\Theta(g ug^{-1})=b_0\}
    \]
    is compact. This is almost immediate by in \cite[Corollary~6.2]{guichardpositivity} but we present its proof. Since $\Theta= \Delta$, we know that $\mathsf{L}_\Theta ^0 = \exp (\mathfrak{a})$. Let $g=\exp (X)\in \exp(\mathfrak{a})$. Then we have
    \[
    f_\Theta(g ug^{-1}) = (e^{\theta_1(X)}, 2 e^{\theta_2(X)},\cdots, r e^{\theta_r(X)})=(1,2,\cdots,r).
    \]
    This shows that 
    \[
    \{g\in \mathsf{L}_\Theta ^0 \,|\, f_\Theta(g ug^{-1})=b_0\}=\{1\}.
    \]

    To show that $[V,\mathsf{L}_\Theta ^0]=1$, let $g\in V$. Since $g$ is in the maximal compact subgroup  $K$ of $\mathsf{G}$, and since  $g \mathsf{L}_\Theta^0 g^{-1} = \mathsf{L}_\Theta ^0=\exp(\mathfrak{a})$, we know that $g$ lies in $N_K(\mathfrak{a})$. Thus, $g$ is an element of the Weyl group and by our choice of $b_0$,  $g$ must fix each coordinate of $f_\Theta(u)$. It follows that $g$ is indeed in the centralizer of $\mathfrak{a}$ in $\mathsf{L}_\theta\cap K$. Hence, $g$ commutes with any elements in $\mathsf{L}_\Theta ^0$. 
\end{proof}

Since $\Sp_{2n}(\bR)$-maximal representations are also positive, we can find the positive limit maps. We will give a concrete description of the associated flag manifold $\mathcal{F}_\Theta$ and limit maps. We follow closely to the work of Burger--Iozzi--Wienhard \cite{burger2010}. 

For the positive structure $\Theta=\{\theta_n\}$ of $\Sp_{2n}(\bR)$, the flag manifold $\mathcal{F}_\Theta$ is so-called the Shilov boundary, which can be identified with the space $\mathcal{L}(\bR^{2n})$ of Lagrangians in $\bR^{2n}$. 

Given an $(n\times n)$-symmetric matrix $A=(a_{ij})$, denote by $\mathcal{L}_A$ the Lagrangian spanned by vectors 
\[
\mathbf{e}_{n+1}+\sum_{i=1} ^n a_{i1}\mathbf{e}_{i}, \quad \mathbf{e}_{n+2}+\sum_{i=1} ^n a_{i2}\mathbf{e}_{i}, \quad\cdots\quad, \mathbf{e}_{2n}+\sum_{i=1} ^n a_{in}\mathbf{e}_{i},
\]
where $\mathbf{e}_i$ are the standard unit vectors in $\bR^n$. By convention we set $\mathcal{L}_\infty$ to be the Lagrangian $\Span(\mathbf{e}_{1}, \mathbf{e}_{2},\cdots,\mathbf{e}_{n})$. 

We say that a triple of Lagrangians $(L_1,L_2,L_3)$ is \emph{maximal} if there is $g\in \Sp_{2n}(\bR)$ such that $(gL_1,gL_2,gL_3) = (\mathcal{L}_O, \mathcal{L}_I, \mathcal{L}_\infty)$. See \cite{burger2017} for more intrinsic definition of maximality. The limit map associated to a maximal representation maps positive triples in $\partial_\infty\pi_1(S)$ to maximal triples in $\mathcal{L}(\bR^{2n})$.  

\begin{thm}[\cite{burger2010}]
    Fix an  orientation of $\partial_{\infty}\pi_1(S)$. A representation $\rho:\pi_1(S) \to \Sp_{2n}(\bR)$ is maximal if and only if there is a $\rho$-equivariant left continuous map $\xi_\rho: \partial_\infty \pi_1(S) \to\mathcal{L}(\bR^{2n})$ which maps positive triples to maximal triples. 
\end{thm}

Given two transverse Lagrangians $\mathcal{L}_O$ and $\mathcal{L}_\infty$, the space
\[
\{L\in \mathcal{L}(\bR^{2n})\,|\,L\text{ is transverse to } \mathcal{L}_O\text{ and }\mathcal{L}_\infty\}
\]
can be identified with the symmetric space $\mathsf{GL}_n(\bR)/\mathsf{O}_n$, or more concretely, the space of congruence classes of $(n\times n)$-symmetric positive definite matrices. Let $U_M$ be the unit ball around the identity matrix;
\begin{equation}\label{eq:UM}
U_M := \left\{ A\, \left|\,A\text{ is symmetric, positive definite and } \|A\|^2=\sum_{i=1} ^n \log^2 \sigma_i(A)\le 1\right\}\right., 
\end{equation}
where $\sigma_i(A)$ is the $i$th largest eigenvalue of $A$. 

We want to establish a result similar to Lemma~\ref{lem:normalization} for maximal representations. For this, we recall some definitions regarding the Jordan canonical form. 

Let $\lambda$ be a real number. A \emph{real Jordan block} $J(\lambda)$ with eigenvalue $\lambda$ is a square matrix of the form 
\[
\begin{pmatrix}
    \lambda & 1 & 0 &0 &\cdots & 0 \\
    0 & \lambda & 1 & 0 &\cdots & 0 \\
    0&0 & \ddots &\ddots & \cdots & 1 \\ 
    0 & 0 & 0 & \cdots  & 0 & \lambda
\end{pmatrix}.
\]
More generally, for a complex number $\lambda=a+\sqrt{-1}b$, a real Jordan block $J(\lambda)$ is a square matrix 
\[
\begin{pmatrix}
    K & I_2 & 0 &0 &\cdots & 0 \\
    0 & K & I_2 & 0 &\cdots & 0 \\
    0&0 & \ddots &\ddots & \cdots & I_2 \\ 
    0 & 0 & 0 & \cdots  & 0 &K
\end{pmatrix}
\]
where $K= \begin{pmatrix}
    a & -b \\ b & a
\end{pmatrix}$ and $I_2= \begin{pmatrix}
    1 &0 \\ 0 & 1
\end{pmatrix}$. 

An \emph{ordered real Jordan form} is a block diagonal matrix 
\[
\operatorname{diag}(J(\lambda_1),J(\lambda_2), \cdots, J(\lambda_k)),
\]
where $J(\lambda_i)\in \mathsf{GL}_{n_i}(\bR)$ are real Jordan blocks such that $|\lambda_1| \ge |\lambda_2| \ge\cdots \ge |\lambda_k| $. 

\begin{lem}\label{lem:normalizationmax}
    Let $\rho\in \Hom_{\mathrm{M}}(\pi_1(S),\Sp_{2n}(\bR))$ and let $\xi_\rho$ be the associated limit map. Let $\gamma\in \pi_1(S)\setminus\{1\}$ and let $\gamma^+$ and $\gamma^-$ be the attracting and repelling fixed points of $\gamma$ in $\partial_\infty\pi_1(S)$. Let $z\in \partial_\infty \pi_1(S)$ be such that $(\gamma^+, z, \gamma^-)$ is positively ordered. Then, there is $g\in \Sp_{2n}(\bR)$ such that 
    \begin{itemize}
         \item[(i)] $( g\xi_\rho (\gamma^+), g\xi_\rho (z),g\xi_\rho(\gamma^-))=  (\mathcal{L}_O, \mathcal{L}_A, \mathcal{L}_\infty)$  for some $A\in U_M$.
         \item[(ii)] $g\rho(\gamma)g^{-1}$ is of the form
         \[
         \begin{pmatrix}
             \mathcal{J} & O \\ O & (\mathcal{J}^{\mathsf{T}})^{-1}
         \end{pmatrix}
         \]
         for some  ordered real Jordan form $\mathcal{J}\in \mathsf{GL}_n(\bR)$.
    \end{itemize}
    Moreover, suppose that $\rho$ satisfies (i) and (ii) and $\rho(\gamma)$ is loxodromic. Then the set
    \[
V_\rho:=\{g\in \Sp_{2n}(\bR)\,|\, g \rho g^{-1} \text{ satisfies (i) and (ii)}\}.
\]
is compact in $\Sp_{2n}(\bR)$ and is abelian. 
\end{lem}

\begin{proof}
We know that there is $g_1\in \Sp_{2n}(\bR)$ such that 
\[
(g_1 \xi_\rho (\gamma^+), g_1 \xi_\rho (z),g_1 \xi_\rho(\gamma^-))=  (\mathcal{L}_O, \mathcal{L}_I, \mathcal{L}_\infty).
\]
Hence $g_1 \rho(\gamma) g_1$ is of the form
\[
\begin{pmatrix}
    X & O \\ O & (X^{\mathsf{T}})^{-1}
\end{pmatrix}
\]
for some $X\in \mathsf{GL}_n(\bR)$ whose eigenvalues are bigger than 1. Find $Y\in \mathsf{GL}_n(\bR)$ such that $YXY^{-1}$ is an ordered real Jordan form $\mathcal{J}$ and $\sum _{i=1} ^n \log ^2 \sigma_i (Y) \le 1$. Conjugate $g_1\rho g_1 ^{-1}$ by 
\[
g_2:=\begin{pmatrix}
    Y & O \\ O & (Y^{\mathsf{T}})^{-1}
\end{pmatrix}
\]
so that $\rho_2:=g_2g_1 \rho(\gamma) (g_2g_1)^{-1}$ satisfies (ii). Since we have $\xi_{\rho_2}(z) = \mathcal{L}_{Y Y^{\mathsf{T}}}$, $g_2g_1$ is the desired element in $\Sp_{2n}(\bR)$. 

Now we want to show that the set $V=V_\rho$ is compact and abelian if $\rho$ satisfies (i) and (ii) and $\rho(\gamma)$ is loxodromic. Since 
\[
\rho(\gamma) = \rho(\gamma)=\begin{pmatrix} \mathcal{J} & O \\ O & \mathcal{J}^{-1}\end{pmatrix}
\]
for a real ordered Jordan form $\mathcal{J}$, we know that any element $g$ in $V$ is of the form
\[
\begin{pmatrix}
    X & O \\ O & (X^{\mathsf{T}})^{-1}
\end{pmatrix}
\]
for some $X\in \mathsf{GL}_n(\bR)$ such that $X\mathcal{J}X^{-1} =\mathcal{J}$. Since $\rho(\gamma)$ is loxodromic, we know that $\mathcal{J}$ is a diagonal matrix. Therefore, $X$ is also a diagonal matrix. This shows that $V$ is abelian. Let $\xi_\rho(z) = \mathcal{L}_A$ for some $\|A\|\le 1$. Because $\xi_{g\rho g^{-1}}(z) = \mathcal{L}_{X A X^{\mathsf{T}}}$, we have $\|XAX^{\mathsf{T}}\|\le 1$. Hence, $V$ can be identified with
\[
V = \{X\in\mathsf{GL}_n(\bR)\,|\,   X\mathcal{J}X^{-1}=\mathcal{J} \} \cap \{X\in \mathsf{GL}_n(\bR)\,|\,\|XAX^{\mathsf{T}}\|\le 1\}.
\]
The former is closed, and the latter is compact. Therefore, $V$ is compact as we wanted. 
\end{proof}

\subsection{Action of the mapping class group}\label{sec:MODaction}
For a compact orientable surface $S$  with possibly non-empty boundary, its \emph{mapping class group} $\Mod(S)$ is defined by the group of isotopy classes of self-homeomorphisms that preserve an orientation and each boundary component set-wise: 
\[
\Mod(S):=  \pi_0(\{\phi\in \operatorname{Homeo}^+(S) \,|\,\phi(\zeta) =\zeta \text{ for each component }\zeta\subset \partial S\}).
\]
Note that this definition coincides with the one used in \cite{mirzakhani2007} but is slightly different from that of \cite{farb}. Indeed, so-called the mapping class group in \cite{farb} is the extension of $\Mod(S)$ by the group generated by Dehn twists along peripheral curves.

Given a subset $B$ of the set of boundary components of $S$, we sometimes consider the group
\[
\Mod^*(S, B) := \pi_0 (\{\phi\in \operatorname{Homeo}^+(S)\,|\,\phi(\zeta)=\zeta,\,\zeta\in B\}). 
\]
We write $\Mod^*(S) = \Mod^*(S,\emptyset)$.

Let  $\phi\in \Mod^*(S)$. We choose an orientation preserving homeomorphism  $f:S\to S$ in the class $\phi$. Then $f$ induces an isomorphism $\pi_1(S,x_0) \to \pi_1(S,f(x_0))$. By post-composing an inner automorphism identifying $\pi_1(S,f(x_0))$ with $\pi_1(S,x_0)$, we get an automorphism $f_\sharp:\pi_1(S,x_0) \to \pi_1(S,x_0)$ depending only on the class $\phi$ up to inner automorphisms. The association $\mathcal{N}: \phi \mapsto f_\sharp$ is, therefore, a well-defined homomorphism from $\Mod^*(S)$ to $\operatorname{Out}(\pi_1(S,x_0))$. The homomorphism $\mathcal{N}$ is injective by the Dehn–Nielsen–Baer theorem.  In this fashion, we implicitly regard $\Mod^*(S)$  as a subgroup of $\operatorname{Out}(\pi_1(S))$.

An automorphism $\phi$ on $\pi_1(S)$ induces the diffeomorphism $\phi_*:\Hit_\mathsf{G}(S) \to \Hit_\mathsf{G}(S)$ given by $\phi_*([\rho]) = [\rho\circ \phi^{-1}]$. Because any inner automorphism induces the identity map on $\Hit_\mathsf{G}(S)$, we know that the outer automorphism group, and therefore $\Mod^*(S)$, can act on $\Hit_{\mathsf{G}}(S)$. When $S$ is closed, this $\Mod(S)$-action on $\Hit_\mathsf{G}(S)$ is proper due to \cite[Proposition 6.3.4]{Lab08} and the fact that every $\mathsf{G}$-Hitchin representation is projectively Anosov. This defines the projection
\begin{equation*}
\Pi_\mathsf{G}: \Hit_{\mathsf{G}}(S) \ra \mathcal{M}_{\mathsf{G}}(S):= \Hit_{\mathsf{G}}(S)/\Mod(S).
\end{equation*}

Now suppose that $S$ has a non-empty boundary. As alluded in \cite{LM09}, one can show that $\Mod^*(S)$ acts properly on $\Hit_{\mathsf{G}}(S)$ by using the doubling argument. More precisely, let $\widehat{S}=S\cup_{\partial S} S$ be the topological double of $S$ and regard $S$ as a subsurface of $\widehat{S}$. Given  $[\rho]\in\Hit_{\mathsf{G}}(S)$, one can find a Hitchin representation $\widehat{\rho}:\pi_1(\widehat{S})\to\mathsf{G}$ such that $\widehat{\rho}|\pi_1(S) = \rho$ called the doubling of $\rho$ (see \cite[Corollary 1.3]{LM09} for $\mathsf{G}=\PSL_n(\bR)$ and \cite[Proposition 2.26]{alessandrini} for general $\mathsf{G}$).  This doubling map embeds $\Hit_n(S)$ into $\Hit_{\mathsf{G}}(\widehat{S})$. One can also embed $\Mod^*(S)$ into $\Mod(\widehat{S})$ as a subgroup by sending a representative $\phi$ of $[\phi]\in \Mod^*(S)$ to its double $\widehat{\phi}$. Since $\Mod(\widehat{S})$ acts properly on $\Hit_{\mathsf{G}}(\widehat{S})$, $
\Mod^*(S)$ also acts properly on $\Hit_{\mathsf{G}}(\widehat{S})$ preserving the embedded subspace $\Hit_{\mathsf{G}}(S)$ of $\Hit_{\mathsf{G}}(\widehat{S})$.  Therefore, we have the map
\begin{equation*}
\Pi_{\mathsf{G}}^*:\Hit_{\mathsf{G}}(S) \ra \mathcal{M}_{\mathsf{G}}(S):=\Hit_{\mathsf{G}}(S)/\Mod^*(S).
\end{equation*}
Moreover, because elements of $\Mod(S)$ preserve each boundary component of $S$, each subspace $\Hit_{\mathsf{G}}(S,\mathbf{B})$ is invariant under the $\Mod(S)$-action for any choice of $\mathbf{B}$.  This allows us to consider the quotient
\[
\Pi_{\mathsf{G},\mathbf{B}}: \Hit_{\mathsf{G}}(S, \mathbf{B}) \to \mathcal{M}_{\mathsf{G}}(S, \mathbf{B}) := \Hit_{\mathsf{G}}(S, \mathbf{B}) /\Mod(S).
\]

To avoid cumbersome notations, denote $\mathcal{M}_{\PSL_n(\bR)}(S)$ and $\mathcal{M}_{\PSL_n(\bR)}(S,\mathbf{B})$ simply by $\mathcal{M}_{n}(S)$ and $\mathcal{M}_{n}(S,\mathbf{B})$. Moreover, unless it causes any confusion, we will use the notation $\Pi$ instead of $\Pi_{\mathsf{G}}$, $\Pi_{\PSL_n(\bR)}^*$, and $\Pi_{\PSL_n(\bR),\mathbf{B}}$. 

On the space of maximal representations, we have the same mapping class group action. Since maximal representations are projectively Anosov, we know that the mapping class group action on $\Max_{2n}(S)$ is proper \cite{wienhard06}. Hence, we get the moduli space $\Max_{2n}(S)/\Mod(S)$. 

We wrap up this subsection by mentioning the following general fact:
\begin{lem}\label{emptyinterior}
Let $\Gamma$ be a countable group acting properly and effectively on a connected manifold $M$. Let $U$ be any non-empty open set of $M$. Then $U$ contains an element having the trivial stabilizer. In particular, there is a non-empty open set $U'\subset U$ such that
\[
\{\gamma\in \Gamma\,|\,\gamma(U') \cap U' \ne \emptyset \} = \{1\}.
\]
\end{lem}
\begin{proof}
Since the $\Gamma$-action is proper and effective,
\[
\operatorname{Fix}(\gamma):=\{x\in M \,|\, \gamma\cdot x = x \},
\]
for $\gamma\in\Gamma\setminus\{1\}$ is nowhere dense. Therefore, 
\[
\bigcup_{\gamma\in \Gamma\setminus \{1\}} \operatorname{Fix}(\gamma)
\]
is nowhere dense. Hence, there is a point $x\in U$ that has the trivial stabilizer in $\Gamma$. The claim for the existence of $U'$ is a consequence of the slice theorem.
\end{proof}

\subsection{Symplectic structure} 
Given a oriented closed surface $S$ of genus $\ge 2$, $\Hit_{\mathsf{G}}(S)$ and $\Max_{2n}(S)$ admit $\Mod(S)$-invariant symplectic structure, called the Atiyah-Bott-Goldman symplectic form $\omega_{ABG}$. 

Suppose that $S$ is a compact oriented surface with oriented boundary components $\zeta_1, \cdots, \zeta_b$. We regard them as elements of $\pi_1(S)$ by taking paths from $x_0$ to a point of $\zeta_i$. Let $\mathbf{B}=(B_1, \cdots, B_b)$ be an ordered set of $\mathsf{G}$-conjugacy classes of loxodromic elements. As in the introduction, the \emph{relative Hitchin component} $\Hit_\mathsf{G}(S, \mathbf{B})$ with respect to the boundary holonomy $\mathbf{B}$ is given by
\[
\Hit_\mathsf{G}(S, \mathbf{B})=\{[\rho] \in \Hit_\mathsf{G}(S)\,|\, \rho(\zeta_i)\in B_i,\,i=1,2,\cdots,b\}.
\]
We equip $\Hit_\mathsf{G}(S,\mathbf{B})$ with the Guruprasad-Huebschmann-Jeffrey-Weinstein form. Note, however, that $\Hit_{\mathsf{G}}(S)$ itself does not carry a meaningful symplectic form. 

In any cases, we denote by $\vol$ the induced volume form on $\Hit_\mathsf{G}(S)$, $\Max_{2n}(S)$ or $\Hit_\mathsf{G}(S, \mathbf{B})$.

\begin{rem}
    After Goldman's work, there have been plenty of constructions for the symplectic and Poisson structures on character varieties \cite{huebschmann95,HJ,Karshon,weinstein,alekseev,boalch}. See also Section~9 of \cite{huebschmann22}. 
\end{rem}

For some split real form  semi-simple Lie group  Lie group $\mathsf{G}$, the $\mathsf{G}$-Hitchin component $\Hit_{\mathsf{G}}(S)$ can be symplectically immersed in $\Hit_d(S)$. We will not use this fact but it may be independently interesting. 

Let 
\[
J_d = \begin{pmatrix}
    0 & 0 & \cdots & 0& 1 \\
    0 & 0 & \cdots  & 1 & 0 \\
    0& 0& \iddots & 0 & 0\\
    0& 1 & \cdots &0 &0\\
    1& 0 & \cdots &0 & 0
\end{pmatrix}\text{, }\quad  W_{2n} = \begin{pmatrix}
0 & I_n \\
-I_n & 0
\end{pmatrix},\quad\text{ and }\quad W_{2n}':=\begin{pmatrix}
    0& J_n \\ -J_n & 0
\end{pmatrix}.
\]
The matrix $J_d$ has the signature $(n,n)$ for $d=2n$ and $(n+1,n)$ for $d=2n+1$. Hence we define the indefinite inner products  
\[
\langle v,w\rangle_{n+1,n}:=v^{\mathsf{T}} J_{2n+1} w,\qquad \text{and}\qquad \langle v,w\rangle_{n,n} :=  v^{\mathsf{T}} J_{2n} w
\]
respectively. On the other hand, $W_{2n}$ and $W_{2n}'$ define symplectic structures 
\[
\omega(v,w) : = v^{\mathsf{T}} W_{2n} w,\qquad \text{and}\qquad \omega'(v,w) = v^{\mathsf{T}}W'_{2n} w
\]
on $\bR^{2n}$. For a moment, we adopt the following definitions for the indefinite orthogonal groups and symplectic groups 
\begin{align}
    \SO_{n+1,n} &= \{g\in \SL_{2n+1}(\bR)\,|\,g^{\mathsf{T}} J_{2n+1}g=J_{2n+1}\} \label{eq:altdefso}\\
    \Sp_{2n}(\bR)&= \{g\in \SL_{2n}(\bR)\,|\,g^{\mathsf{T}} W'_{2n}g=W'_{2n}\} \label{eq:altdefsp}.
\end{align}

Having the above definitions on hand, we construct a natural map from $\Hit_{\mathsf{G}}(S)$ into $\Hit_d(S)$.  In what follows, let $\mathsf{G}=\PSO_{n+1,n}$ or $\PSp_{2n}(\bR)$ and $d=2n+1$ or $2n$ depending on $\mathsf{G}$. Let $\iota_{\mathsf{G},d}:\mathsf{G} \to \PSL_{d}(\bR)$ be the defining representation. Denote by the same notation 
\[
\iota_{\mathsf{G},d}:\Hit_\mathsf{G}(S)\to \Hom(\pi_1(S),\PSL_d(\bR))/\PSL_d(\bR)
\]
the induced map defined by $\iota_{\mathsf{G},d}([\rho]) = [\iota_{\mathsf{G},d} \circ \rho]$. 

\begin{lem}\label{lem:smoothembedding}
Let $\mathsf{G}=\PSO_{n+1,n}$ or $\PSp_{2n}(\bR)$. For any $[\rho]\in \Hit_\mathsf{G}(S)$, we have $\iota_{\mathsf{G},d}([\rho])\in \Hit_d(S)$. Hence $\iota_{\mathsf{G},d}$ induces the well-defined smooth map $\iota_{\mathsf{G},d}: \Hit_\mathsf{G}(S)\to \Hit_d(S)$.
\end{lem}
\begin{proof}
The smoothness of $\iota_{\mathsf{G},d}$ follows from the fact that $\iota:\mathsf{G}\to \PSL_d(\bR)$ is smooth and that $\Hit_d(S)$ and $\Hit_{\mathsf{G}}(S)$ are smooth manifolds \cite{hitchin92,fock}.

We claim that $\iota_{\mathsf{G},d}\circ \tau_{\mathsf{G}}=\tau_{\PSL_d(\bR)}$, where $\tau_{\mathsf{G}}$ and $\tau_{\PSL_d(\bR)}$ are preferred representations defined in Section~\ref{sec:notation}.  To see this, we need to show that a principal TDS of $\mathfrak{so}_{n+1,n}$ (or of $\mathfrak{sp}_{2n}(\bR)$) is also a principal TDS of $\mathfrak{sl}_{2n+1}(\bR)$ (or of $\mathfrak{sl}_{2n}(\bR)$, respectively). These are consequences of our definitions of $\PSO_{n+1,n}$ and $\PSp_{2n}(\bR)$ given in (\ref{eq:altdefso}), (\ref{eq:altdefsp}). For instance, in the case of $\mathsf{G}=\PSO_{n+1,n}$, we may choose 
\begin{align*}
E&=\sum_{k=1}^n  k (E_{k,k+1}-E_{2n+1-k,2n+2-k}),\\
F&=\sum_{k=1}^n (2n+1-k)(E_{k+1,k}-E_{2n+2-k,2n+1-k}),\\
H&= \sum_{k=1}^n (2n+2-2k)(E_{k,k}-E_{2n+2-k,2n+2-k}),
\end{align*}
in $\mathfrak{so}_{n+1,n}$ and if $\mathsf{G}=\PSp_{2n}(\bR)$, we choose
\begin{align*}
E&=\sum_{k=1}^{n-1}  k (E_{k,k+1}-E_{2n-k,2n+1-k})+nE_{n,n+1},\\
F&=\sum_{k=1}^{n-1} (2n+1-k)(E_{k+1,k}-E_{2n+1-k,2n-k})+(n+1)E_{n+1,n},\\
H&= \sum_{k=1}^n (2n+2-2k)(E_{k,k}-E_{2n+1-k,2n+1-k})
\end{align*}
in $\mathfrak{sp}_{2n}(\bR)$, where $E_{i,j}$ is a matrix (of suitable size) whose $(i,j)$ entry equals 1 and all other entries are 0. Proving that they are both principal TDSs of $\mathfrak{sl}_d(\bR)$ is a straightforward computation. 

    Let $[\rho]\in\Hit_{\mathsf{G}}(S)$ be given. We can find a continuous path $[\rho_t]$ in $\Hit_{\mathsf{G}}(S)$ such that $[\rho_0]=[\rho]$ and $[\rho_1]=[\tau_{\mathsf{G}}\circ \rho_F]$ for some Fuchsian $[\rho_F]$. By the above observation, $[\iota_{\mathsf{G},d} \circ \rho_t]$ is a path from $[\iota_{\mathsf{G},d}\circ\rho]$ to $[\tau_{\PSL_d(\bR)}\circ\rho_F]$ in $\Hom(\pi_1(S),\PSL_d(\bR))/\PSL_d(\bR)$. By definition, we know that $\iota_{\mathsf{G},d}([\rho])=[\iota_{\mathsf{G},d}\circ\rho]\in \Hit_d(S)$. 
\end{proof}

We saw in Lemma~\ref{lem:smoothembedding} that the defining representation $\mathsf{G}\to \PSL_d(\bR)$ sometimes induces the smooth map $\Hit_{\mathsf{G}}(S)\to \Hit_d(S)$. We claim that this map is better than just a smooth map. 

\begin{lem}\label{lem:embedding}
Let $\mathsf{G}= \PSO_{n+1,n}$ or $\PSp_{2n}(\bR)$.  The map $\iota_{\mathsf{G},d}: \Hit_\mathsf{G}(S)\to \Hit_d(S)$ given in Lemma~\ref{lem:smoothembedding} preserves the symplectic structures and is an immersion.
\end{lem}
\begin{proof}
Recall that the Atiyah-Bott-Goldman symplectic forms are defined pointwise by the cup product on the tangent spaces $H^1(\pi_1(S); \mathfrak{g}_\rho)$ and $H^1(\pi_1(S);\mathfrak{sl}_d(\bR)_{\iota_{\mathsf{G},d} \circ \rho})$. By naturality of the cup product, we know that $\iota_{\mathsf{G},d}$  preserves the Atiyah-Bott Goldman symplectic forms. 

Now we show that $\iota_{\mathsf{G},d}$ is an immersion. Let $\mathfrak{g}_\rho$ be the Lie algebra of $\mathsf{G}$ regarded as a $\pi_1(S)$-module via the adjoint action, i.e., $\gamma\cdot X = \operatorname{Ad}_{\rho(\gamma)}(X)$. Since the tangent space of $\Hit_\mathsf{G}(S)$ at $[\rho]$ can be identified with the twisted group cohomology $H^1(\pi_1(S);\mathfrak{g}_\rho)$, we need to show that the induced map
\[
\operatorname{d}\iota_{\mathsf{G},d}: H^1(\pi_1(S);\mathfrak{g}_\rho) \to H^1(\pi_1(S); \mathfrak{sl}_d(\bR)_{\iota_{\mathsf{G},d} \circ \rho})
\]
between the tangent spaces is injective for all $[\rho]\in \Hit_\mathsf{G}(S)$. This is also equivalent to showing that $H^0(\pi_1(S);\mathfrak{sl}_{d}(\bR)_{\iota_{\mathsf{G},d} \rho}/\mathfrak{g}_\rho)=0$. 

We only deal with the $\mathsf{G}=\PSO_{n+1,n}$ case because the same argument applies to the  $\mathsf{G}=\PSp_{2n}(\bR)$ case. Suppose that there is $X\in \mathfrak{sl}_{d}(\bR)_{\iota_{\mathsf{G},d}\circ \rho}$ such that  $\gamma\cdot X - X \in \mathfrak{g}$ for all $\gamma\in \pi_1(S)$. Note that 
\[
\mathfrak{g} = \{ X\in \mathfrak{sl}_d(\bR)\,|\,X^* = -X\},
\]
where  $X^*:= J_d X^{\mathsf{T}} J_d$. Thus, we have that $(\gamma \cdot X - X )^* = - \gamma \cdot X +X $. This implies  $\gamma\cdot (X+X^*) = X+X^*$ for all $\gamma\in \pi_1(S)$. Since $\iota_{\mathsf{G},d}\circ \rho$ is Hitchin, we have $H^0(\pi_1(S);\mathfrak{sl}_{d}(\bR)_{\iota_{\mathsf{G},d}\circ \rho})=0$. Thus, we have $X=-X^*$. It follows that $X\in \mathfrak{g}$ as we wanted. 
\end{proof}

\section{Goldman flows}\label{sec:goldmanflow}
In this section, we discuss decompositions of the Hitchin and maximal components induced from a splitting of the underlying surface. We also recall algebraic bending, Goldman flow and bulging flow.

We say a simple closed curve is \emph{essential} if it is neither freely homotopic to a point nor a boundary component. Given an oriented closed curve $\alpha$, we use the notation $|\alpha|$ to denote the unoriented closed curve as a subset of $S$ forgetting the orientation of $\alpha$. 

\subsection{Decomposition of  Hitchin and maximal components}\label{sec:decomposition}
Throughout this subsection,  $S$ is a compact orientable surface of negative Euler characteristic not homeomorphic to a sphere with three boundary components, and $\alpha$ is an oriented essential simple closed curve in $S$. 

We decompose the fundamental group $\pi_1(S)$ in terms of an amalgamated product or a HNN extension depending on whether $\alpha$ is separating or not. We also describe the effect of the decomposition of the fundamental group on the Hitchin components. 

\subsubsection{When $\alpha$ is non-separating} \label{fungpdecomnonsep}
In this case, $S\setminus |\alpha|$ is connected. Choose a tubular neighborhood $\nu(|\alpha|)$ of $|\alpha|$ and fix an orientation of $S$. Let $S_1=S\setminus \nu(|\alpha|)$. We have the natural inclusion map $j:S_1\to S$.

For each point $x$ in the image of $\alpha$, we have a tangent vector $t_x$ of $\alpha$. We choose a normal vector $n_x$ such that $(t_x, n_x)$ forms a positive basis of $T_xS$. Since $S$ is oriented, $\nu(|\alpha|)\setminus|\alpha|$ consists of two connected components. We denote by $C_+$ a connected component of $\nu(|\alpha|)\setminus|\alpha|$ to which $n_x$ points inward. Let $C_-$ be the other connected component of  $\nu(|\alpha|)\setminus|\alpha|$. The surface $S_1$ has two distinguished boundary components, say $\zeta^+=S_1\cap \overline{C_+}$, $\zeta^-=S_1\cap \overline{C_-}$, where $\overline{C_\pm}$ are the closures of $C_\pm$ in $S$.

Fix a base-point $x_1\in \zeta^+\subset S_1$ of $S_1$. Then the map $j:S_1\to S$ induces a homomorphism $j_\sharp:\pi_1(S_1, x_1) \to \pi_1(S,x_0)$.  We can find elements $\alpha_{x_0}\in \pi_1(S, x_0)$, and  $\zeta^\pm _{x_1}\in \pi_1(S_1, x_1)$ that are freely homotopic to $\alpha$, and $\zeta^\pm$ respectively, such that 
\[
\alpha_{x_0} = j_{\sharp}(\zeta^+_{x_1}),\quad \text{and}\quad v^{-1} j_\sharp(\zeta^-_{x_1})v = j_\sharp(\zeta^+_{x_1})
\]
for some $v\in \pi_1(S,x_0)$.  From now on, we write $\pi_1(S)$ for $\pi_1(S,x_0)$, $\pi_1(S_1)$ for $\pi_1(S_1, x_1)$ provided $x_0$ and $x_1$ are understood. Consider the HNN-extension
\[
\pi_1(S_1) \ast_{\langle\zeta^+ _{x_1}\rangle,\langle\zeta^- _{x_1}\rangle}:=(\pi_1(S_1)\ast \langle v\rangle)/R
\]
for a new generator $v$ and the normal subgroup $R$ generated by $v\zeta^+ _{x_1} v^{-1}(\zeta^- _{x_1})^{-1}$. Under these choices and conventions, we have an isomorphism 
\[
\mathfrak{j} : \pi_1(S_1)*_{\langle\zeta^+ _{x_1}\rangle,\langle\zeta^- _{x_1}\rangle}\to \pi_1(S).
\]
This isomorphism, in particular, maps the $\pi_1(S_1)$ factor to a subgroup $j_{\sharp}(\pi_1(S_1,x_1))$ of $\pi_1(S,x_0)$

For a split real form $\mathsf{G}$, we have the restriction map
\[ {\widetilde{\mathcal{R}}}_1 : \Hom_{\mathrm{H}}(\pi_1(S), \mathsf{G}) \ra \Hom_{\mathrm{H}}(\pi_1(S_1), \mathsf{G})\]
defined by sending $\rho$ to $\rho| j_\sharp(\pi_1(S_1))$. By the universal property of the HNN extension, the map $\widetilde{\mathcal{R}}_1$ has the image
\[
\widetilde{\Delta}_\alpha = \{\rho\,|\, \rho(\zeta^+_{x_1}) \text{ and } \rho(\zeta^-_{x_1})\text{ are conjugate in } \mathsf{G}\}
\]
regarded as a subspace of $\Hom_{\mathrm{H}}(\pi_1(S_1),\mathsf{G})$.

Note that $\widetilde{\mathcal{R}}_1$  commutes with the $\mathsf{G}$-conjugation actions, inducing the continuous map
\[
\mathcal{R}_1: \Hit_\mathsf{G}(S) \to \Delta_\alpha \subset\Hit_\mathsf{G}(S_1)
\]
where
\[\Delta_\alpha:=\{ [\rho]\in \Hit_\mathsf{G}(S_1)\,|\,[\rho(\zeta^+_{x_1})]=[\rho(\zeta^-_{x_1})]\}. \]

If we fix the boundary holonomy $\mathbf{B}=(B_1,\cdots, B_b)$, we also have the same restriction map $\mathcal{R}_{1}:\Hit_\mathsf{G}(S, \mathbf{B}) \to \Delta_\alpha(\mathbf{B})$, where
\[
\Delta_\alpha(\mathbf{B}) = \{[\rho]\in \Hit_\mathsf{G}(S_1)\,|\,[\rho(\zeta_{x_1}^+)]=[\rho(\zeta_{x_1}^-)],\,\rho(\zeta_i)\in B_i\}.
\]

For the maximal representation case, first recall that the Toledo invariant remains the same if we cut the surface along a non-separating simple closed curve. Therefore, for a given maximal representation $\rho\in \Hom_{\mathrm{M}}(\pi_1(S),\Sp_{2n}(\bR))$ and a non-separating essential simple closed curve $\alpha$,   $\widetilde{\mathcal{R}}_1 (\rho)\in \Hom(\pi_1(S_1),\Sp_{2n}(\bR))$ is a maximal representation. This yields the well-defined continuous map
\[
    \mathcal{R}_1: \Max_{2n}(S) \to \Max_{2n}(S_1),
\]
whose image is $\{ [\rho]\in \Max_{2n}(S_1)\,|\,[\rho(\zeta^+_{x_1})]=[\rho(\zeta^-_{x_1})]\}$.

\subsubsection{When $\alpha$ is separating}\label{fungpdecomsep}
In this case, $S\setminus|\alpha|$ consists of two components. We choose again a tabular neighborhood $\nu(|\alpha|)$ and let $S_1$ and $S_2$ be two connected components of $S\setminus \nu(|\alpha|)$. Let $j_i: S_i \to S$, $i=1,2$, be the natural inclusion maps. 

As before, we denote by $C_+$ the connected component of $\nu(|\alpha|)\setminus|\alpha|$ to which a normal vector points inward. We may assume that $S_1$ is the one that contains $C_+$. Let $\zeta^{(1)}=S_1\cap \overline{C_+}$ and $\zeta^{(2)} = S_2\cap \overline{C_-}$.

We choose a point $x_i\in \zeta^{(i)}$  and use it as a base-point for $S_i$ for each $i=1,2$.  By taking paths joining $x_0$ and $x_i$ for each $i=1,2$, we obtain injective homomorphisms $(j_i)_\sharp:\pi_1(S_i,x_i)\to \pi_1(S,x_0)$. Choose $\alpha_{x_0}\in \pi_1(S,x_0)$, $\zeta^{(1)}_{x_1}\in \pi_1(S_1,x_1)$ and $\zeta^{(2)}_{x_2}\in \pi_1(S_2,x_2)$ freely homotopic to $\alpha$, $\zeta^{(1)}$ and $\zeta^{(2)}$ respectively such that 
\[
\alpha_{x_0} = (j_1)_\sharp (\zeta^{(1)}_{x_1})=(j_2)_\sharp (\zeta_{x_2}^{(2)}). 
\]
Let
\[
\pi_1(S_1) \ast_{\langle \zeta^{(1)} _{x_1}\rangle,\langle\zeta^{(2)}_{x_2}\rangle}  \pi_1(S_2)
\]
be the amalgamated product with amalgam $\zeta^{(1)}_{x_1}=\zeta^{(2)}_{x_2}$. Under these identifications, there is an isomorphism
\[
\mathfrak{j}: \pi_1(S_1) \ast_{\langle\zeta^{(1)} _{x_1}\rangle,\langle\zeta^{(2)}_{x_2}\rangle}  \pi_1(S_2)\to \pi_1(S).
\]
Note that $\mathfrak{j}$ sends each $\pi_1(S_i)$ to the subgroup $(j_i)_\sharp(\pi_1(S_i,x_i))$ of $\pi_1(S,x_0)$.

As in the non-separating case, we have restriction maps
\[ {\widetilde{\mathcal{R}}}_i : \Hom_{\mathrm{H}}(\pi_1(S), \mathsf{G}) \ra \Hom_{\mathrm{H}}(\pi_1(S_i), \mathsf{G}),\qquad i=1, 2\]
defined by sending $\rho$ to $\rho| (j_i)_\sharp(\pi_1(S_i))$. The product map
\[
\widetilde{\mathcal{R}}_1\times \widetilde{\mathcal{R}}_2: \Hom_{\mathrm{H}}(\pi_1(S),\mathsf{G}) \to \widetilde{\Delta}_\alpha,
\]
is a homeomorphism, where
\[
\widetilde{\Delta}_\alpha = \{(\rho_1,\rho_2)\,|\, \rho_1(\zeta^{(1)} _{x_1}) = \rho_2(\zeta^{(2)}_{x_2})\}
\]
regarded as a subspace of 
\[
\Hom_{\mathrm{H}}(\pi_1(S_1),\mathsf{G})\times \Hom_{\mathrm{H}}(\pi_1(S_2),\mathsf{G}).
\]
The inverse of the map $\widetilde{\mathcal{R}}_1\times \widetilde{\mathcal{R}}_2$ is given simply by gluing two representations $(\rho_1,\rho_2)\in\widetilde{\Delta}_\alpha$ along $\alpha_{x_0}$ by using the universal property of the amalgamated product.

Since $\widetilde{\mathcal{R}}_1$ and $\widetilde{\mathcal{R}}_2$ commute with the conjugation action, inducing the continuous maps
\[
{\mathcal{R}}_1\times {\mathcal{R}}_2 : \Hit_\mathsf{G}(S) \to \Delta_\alpha \subset\Hit_\mathsf{G}(S_1) \times \Hit_\mathsf{G}(S_2)
\]
where
\[\Delta_\alpha:=\{ ([\rho_1], [\rho_2])\in \Hit_\mathsf{G}(S_1) \times \Hit_\mathsf{G}(S_2)\,|\, [\rho_1(\zeta^{(1)}_{x_1})]
=[\rho_2(\zeta^{(2)} _{x_2})]\}. \]

Suppose that we fix boundary holonomy $\mathbf{B}=(B_1, \cdots, B_b)$. Assume that boundary components $\zeta_1, \cdots, \zeta_{b_1}$ are contained in $S_1$ and $\zeta_{b_1+1},\cdots, \zeta_b$ are contained in $S_2$. Define
\begin{multline*}
\Delta_{\alpha}(\mathbf{B})=\{ ([\rho_1], [\rho_2])\in \Hit_\mathsf{G}(S_1) \times \Hit_\mathsf{G}(S_2)\,|\, [\rho_1(\zeta^{(1)}_{x_1})]=[\rho_2(\zeta^{(2)} _{x_2})],\\
\rho_1(\zeta_1)\in B_1,\cdots,\rho_1(\zeta_{b_1})\in B_{b_1},\,\rho_2(\zeta_{b_1+1})\in B_{b_1+1},\cdots,\rho_2(\zeta_b)\in B_b\}. 
\end{multline*}
Then we have the same restriction map $\mathcal{R}_{1}\times \mathcal{R}_{2}: \Hit_\mathsf{G}(S,\mathbf{B}) \to \Delta_\alpha (\mathbf{B})$.

Now we discuss the maximal representations. Recall that the Toledo invariant is additive under the boundary sum. Thus, for a given maximal representation $\rho\in \Hom_{\mathrm{M}}(\pi_1(S),\Sp_{2n}(\bR))$ and a separating essential simple closed curve $\alpha$,   $\widetilde{\mathcal{R}}_i(\rho)\in \Hom(\pi_1(S_i), \Sp_{2n}(\bR))$, $i=1,2$ are all maximal representations. Therefore, we have the well-defined map
\[
\mathcal{R}_1\times\mathcal{R}_2: \Max_{2n}(S) \to \Max_{2n}(S_1)\times \Max_{2n}(S_2).
\]
The image of this map is given by  
\[
\{ ([\rho_1], [\rho_2])\in \Max_{2n}(S_1) \times \Max_{2n}(S_2)\,|\, [\rho_1(\zeta^{(1)}_{x_1})]
=[\rho_2(\zeta^{(2)} _{x_2})]\}.
\]

\subsection{Algebraic bending}\label{sec:bending} In the Teichm\"uller space, we have the twisting deformation along an essential simple closed curve. In general  representation varieties, algebraic bending plays the role of the twisting deformation.

Section~5 of Johnson-Millson \cite{JM87} is a general reference for this section. 

\subsubsection{Normalized representations}
To define algebraic bending, we introduce a technical notion that allows us to lift $\Hit_\mathsf{G}(S)$ into $\Hom_{\mathrm{H}}(\pi_1(S), \mathsf{G})$. 

We continue to use the notations from Sections~\ref{fungpdecomnonsep} and ~\ref{fungpdecomsep}. Suppose that $\alpha$ is non-separating. Since $\mathfrak{j}(\pi_1(S_1))\subset \pi_1(S)$ is a quasi-convex subgroup, we can regard $\partial_\infty \pi_1(S_1)$ as a subspace of $\partial_\infty \pi_1(S)$. Fix an orientation on $\partial_\infty \pi_1(S)$. We may assume that for any $x\in \partial_\infty\pi_1(S_1)$, the triple $(\alpha_{x_0}^+,x,\alpha_{x_0}^-)$ is positively ordered.  We choose and fix a point $z\ne\alpha_{x_0}^\pm$ in $\partial_\infty \pi_1(S_1)\subset \partial_\infty \pi_1(S)$. When $\alpha$ is separating, we  assume that the triple $(\alpha_{x_0}^+,x,\alpha_{x_0}^-)$ is positively ordered  for any $x\in \partial_\infty\pi_1(S_1)$. We then choose  $z\ne\alpha_{x_0}^\pm$ in $\partial_\infty \pi_1(S_1)$.  For such a fixed choice of $z$,  a $\mathsf{G}$-Hitchin representation $\rho:\pi_1(S) \to \mathsf{G}$ is  $(\alpha_{x_0},z)$-\emph{normalized} if the positive limit map $\xi_\rho:\partial_\infty\pi_1(S)\to \mathcal{F}_\Theta$ satisfies 
\[
(\xi_\rho (\alpha_{x_0}^+), \xi_\rho (z), \xi_\rho (\alpha_{x_0}^-))\in U_\mathsf{G}.
\]
See Lemma~\ref{lem:normalization} for the definitions of $U_\mathsf{G}$. Indeed, Lemma~\ref{lem:normalization} shows that, for a fixed pair $(\alpha_{x_0},z)$,  the set of $(\alpha_{x_0},z)$-normalized $\Theta$-positive representations in a given conjugacy class is compact.

We also introduce normalized maximal representations. We say a maximal representation is $(\alpha_{x_0}, z)$-\emph{normalized} if its limit map $\xi_\rho:\partial_\infty \partial \pi_1(S) \to \mathcal{L}$ satisfies
    \begin{itemize}
         \item[(i)] $( \xi_\rho (\alpha_{x_0}^+), \xi_\rho (z),\xi_\rho(\alpha_{x_0}^-))=  (\mathcal{L}_O, \mathcal{L}_A, \mathcal{L}_\infty)$  for some $A\in U_M$.
         \item[(ii)] $g\rho(\gamma)g^{-1}$ is of the form
         \begin{equation}\label{eq:standard}         
         \begin{pmatrix}
             \mathcal{J} & O \\ O & (\mathcal{J}^{\mathsf{T}})^{-1}
         \end{pmatrix}
         \end{equation}
         for some  ordered real Jordan form $\mathcal{J}\in \mathsf{GL}_n(\bR)$.
    \end{itemize}
Lemma~\ref{lem:normalizationmax} asserts that  the set of $(\alpha_{x_0}, z)$-normalized representations in a given conjugacy class is compact. 

\subsubsection{Algebraic bending}
Suppose that $\alpha$ is non-separating. Consider the isomorphism 
\[
\mathfrak{j}: \pi_1(S_1)*_{\langle\zeta^+_{x_1}\rangle,\langle\zeta^- _{x_1}\rangle}\to \pi_1(S)
\]
defined in Section~\ref{fungpdecomnonsep}. Let $\rho$ be a  $(\alpha_{x_0},z)$-normalized representation and let 
\[
Z(\rho(\alpha_{x_0})):=\{g\in \mathsf{G}\,|\,g\rho(\alpha_{x_0})g^{-1} = \rho(\alpha_{x_0})\}
\]
be the centralizer of $\rho(\alpha_{x_0})$. We are mostly interested in the case when $\rho(\alpha_{x_0})$ sits inside $\exp (\mathfrak{a})$. If this happens, the identity component of $Z(\rho(\alpha_{x_0}))$ equals $\exp \mathfrak{a}$.

For a given $Q\in \mathfrak{a}$ and a $(\alpha_{x_0},z)$-normalized representation $\rho$, define the set map on $\pi_1(S)\cup \{v\}$ by
\[
\gamma\mapsto \begin{cases}
     \rho(\mathfrak{j}(\gamma)) & \text{ for } \gamma  \in \pi_1(S_1)\\
    \rho(\mathfrak{j}(v))  \exp(Q) &  \text{ for } \gamma =v
\end{cases}
\]
which induces a well-defined new representation 
\[
\widetilde{\Phi}^{Q,\alpha_{x_0}}(\rho): \pi_1(S_1)*_{\langle\zeta^+_{x_1}\rangle,\langle\zeta^- _{x_1}\rangle}\to \mathsf{G}.
\]
By pre-composing $\mathfrak{j} ^{-1}$, we have a representation 
\[
\pi_1(S) \overset{\mathfrak{j}^{-1}}{\to} \pi_1(S_1)*_{\langle\zeta^+_{x_1}\rangle,\langle\zeta^- _{x_1}\rangle}\to \mathsf{G}
\]
denoted by the same notation $\widetilde{\Phi}^{Q,\alpha_{x_0}}(\rho)$. 

By the same fashion, we define the algebraic bending for a separating curve $\alpha$. First, recall that we have the isomorphism 
\[
\mathfrak{j}: \pi_1(S_1) \ast_{\langle\zeta^{(1)} _{x_1}\rangle,\langle\zeta^{(2)}_{x_2}\rangle}  \pi_1(S_2)\to \pi_1(S)
\]
defined in Section~\ref{fungpdecomsep}. For a given $(\alpha_{x_0},z)$-normalized representation $\rho$, define the set map on $\pi_1(S_1)\cup \pi_1(S_2)$ by
\[
\gamma\mapsto\begin{cases}
     \rho(\mathfrak{j}(\gamma)) & \text{ for } \gamma  \in \pi_1(S_1)\\
     (\exp Q ) \rho(\mathfrak{j}(\gamma))\exp (-Q) & \text{ for } \gamma \in \pi_1(S_2)
\end{cases}
\]
which extends to a well-defined representation 
\[
\widetilde{\Phi}^{Q,\alpha_{x_0}}(\rho): \pi_1(S_1) \ast_{\langle\zeta^{(1)} _{x_1}\rangle,\langle\zeta^{(2)}_{x_2}\rangle}  \pi_1(S_2)\to \mathsf{G}.
\]
Precomposing  $\mathfrak{j} ^{-1}$ yields a new representation $\pi_1(S) \to \mathsf{G}$ denoted by the same notation $\widetilde{\Phi}^{Q,\alpha_{x_0}}(\rho)$.

The new representation  $\widetilde{\Phi}^{Q,\alpha_{x_0}}(\rho)$ does not depend on the choice of the isomorphism $\mathfrak{j}$. We also remark that if $\rho$ is $(\alpha_{x_0},z)$-normalized, then $\widetilde{\Phi}^{Q,\alpha_{x_0}}(\rho)$ is also $(\alpha_{x_0},z)$-normalized.

The map $\widetilde{\Phi}^{Q,\alpha_{x_0}}$ descends to the $\mathsf{G}$-Hitchin component as a diffeomorphism in the following sense. Let $Q\in \mathfrak{a}$ and define $\Phi^{Q,\alpha_{x_0}}: \Hit_\mathsf{G}(S) \ra \Hit_\mathsf{G}(S)$ by
\[
\Phi^{Q,\alpha_{x_0}}([\rho]) := [\widetilde{\Phi}^{Q,\alpha_{x_0}}(\rho')]
\]
where $\rho'$ is a $(\alpha_{x_0}, z)$-normalized representation in the class $[\rho]$.  The last assertion of Lemma~\ref{lem:normalization} implies that $\Phi^{Q,\alpha_{x_0}}$ does not depend on the choice of point $z\in \partial_\infty\pi_1(S_1)$ nor the choice of $(\alpha_{x_0},z)$-normalized representative in $[\rho]$. Note also that $\Phi^{Q,\alpha_{x_0}}$ preserves subspaces $\Hit_\mathsf{G}(S,\mathbf{B})$ for any choice of $\mathbf{B}$.

The maximal representation case is more subtle. Let $Q\in \mathfrak{a}$ and let 
\[
\Max_{2n} ^{+,\alpha}(S):=\{[\rho]\in \Max_{2n} ^+(S)\,|\, \rho(\alpha)\text{ is loxodromic}\}.
\]
We define the smooth map on $\Max_{2n} ^{+,\alpha}(S)$ by
\[
\Phi^{Q,\alpha}([\rho]) := [\widetilde{\Phi}^{Q,\alpha_{x_0}}(\rho')]
\]
where $\rho'$ is a $(\alpha_{x_0},z)$-normalized maximal representation in the class $[\rho]$. We remark that this map is not well-define on the whole space $\Max_{2n} ^+(S)$.

\begin{rem}\label{conjugacyinvariance}
    Let $\beta$ and $\gamma$ be two conjugate elements of $\pi_1(S,x_0)$ represented by simple closed curves.  Then $\Phi^{Q,\beta} = \Phi^{Q,\gamma}$ on $\Hit_\mathsf{G}(S)$ and $\Max_{2n}^+(S)$. Hence we may drop $x_0$ from $\alpha_{x_0}$ when we consider the action on $\Hit_\mathsf{G}(S)$ and $\Max_{2n}^+(S)$. However, the action of  $\Phi^{Q,\alpha}$ still depends on the orientation of $\alpha$.
\end{rem}
\subsubsection{Dehn twists and algebraic bending}
The action of the Dehn twists on $\Hit_\mathsf{G}(S)$ can be described in terms of algebraic bending. By convention, Dehn twists always mean left Dehn twists. 

We first treat the Dehn twist $W_{\alpha}$ along a non-separating curve $\alpha$. The induced automorphism $(W_\alpha)_\sharp$ on $\pi_1(S)$ by $W_\alpha$ (up to inner automorphism) is given as follows. We first recall the isomorphism 
\[
\mathfrak{j}:\pi_1(S_0)*_{\langle\zeta^+_{x_1}\rangle,\langle\zeta^- _{x_1}\rangle} \to \pi_1(S)
\]
defined in Section~\ref{fungpdecomnonsep}. By the definition of the Dehn twist, we know that
\[
(W_\alpha)_\sharp(\gamma) = \begin{cases}
    \gamma & \text{ for }\gamma\in \mathfrak{j}(\pi_1(S_0))\\
    \mathfrak{j}(v) \alpha_{x_0} & \text{ for } \gamma =\mathfrak{j}(v)
\end{cases}.
\]
For any $(\alpha_{x_0},z)$-normalized  representation $\rho$, we compute
\[
(W_\alpha)_* (\rho)(\gamma) = \rho ((W_\alpha)_\sharp ^{-1}(\gamma))=\begin{cases}
    \rho(\gamma) & \text{ for }\gamma\in \mathfrak{j}(\pi_1(S_0))\\
    \rho(\mathfrak{j}(v))\rho( \alpha_{x_0})^{-1} & \text{ for } \gamma =\mathfrak{j}(v)
\end{cases}.
\]
The latter is precisely the definition of $\widetilde{\Phi}^{\rho(\alpha_{x_0})^{-1},\alpha_{x_0}}(\rho)$. On the level of $\Hit_\mathsf{G}(S)$, we may write 
\[
(W_\alpha)_* ([\rho]) = \Phi^{T([\rho]),\alpha}([\rho]),
\]
where the function $T$ from $\Hit_\mathsf{G}(S)$ to the group $\mathfrak{a}$ is given by the Jordan projection $\lambda:\mathsf{G}\to \overline{\mathfrak{a}^+}$, i.e., $T([\rho]) = - \lambda(\rho(\alpha_{x_0}))$. In particular, when $\mathsf{G}=\PSL_n(\bR)$, we have 
\[
T([\rho]) = \operatorname{diag}(-\log \lambda_1(\rho(\alpha_{x_0})), -\log \lambda_2(\rho(\alpha_{x_0})), \cdots, -\log \lambda_n(\rho(\alpha_{x_0}))).
\]
Here, $\lambda_i(\rho(\alpha_{x_0}))$ is the $i$th largest eigenvalue of $\rho(\alpha_{x_0})$.

When $\alpha$ is separating, the same argument proves that $(W_\alpha)_* ([\rho]) = \Phi^{T([\rho]),\alpha}([\rho])$ where $T$ is the same function defined above. 

The effect of the Dehn twist on the space of maximal representions can also be interpreted as the algebraic bending. Again, we restrict to the space $\Max_{2n}^{+,\alpha}(S)$ to make our discussion easier. Under this restriction,  the Dehn twist $W_\alpha$ induces the well-defined map
\[
(W_\alpha)_*:[\rho]\mapsto \Phi^{T([\rho]), \alpha}([\rho])
\]
on $\Max_{2n}^{+,\alpha}(S)$, where $T$ is the same map as before. 

\subsubsection{Properness of Algebraic Bending}
It is also well-known that algebraic bending is proper, at least on Hitchin components. One can find a proof for this, for instance, in \cite{CJK20}. The following result covers the maximal representation case as well.

\begin{prop}\label{prop:proper} Let $S$ be a compact orientable surface with negative Euler characteristic. Let $\alpha$ be an essential simple closed curve in $S$. Given any compact subset $N$ of $\Hit_{\mathsf{G}}(S)$, there is a constant $C_N$ such that  $\Phi^{Q,\alpha}([\rho]) \notin N$ for any $[\rho]\in N$ and $Q\in \mathfrak{a}$ with $\|Q\|>C_N$. The same conclusion holds for a closed surface $S$ and a compact set $N$ in $\Max^{+,\alpha} _{2n}(S)$.
\end{prop} 
\begin{proof} 
The proof goes in parallel for both  $\Hit_{\mathsf{G}}(S)$ and $\Max_{2n}^{+,\alpha}(S)$ cases. Hence we just prove for the maximal component case.

Recall that we have fixed $z\in \partial_\infty \pi_1(S_1)$, $z\ne \alpha_{x_0} ^\pm$. Let 
\[
\mathfrak{M}=\{\rho\in \Hom_{\mathrm{M}}^+(\pi_1(S),\Sp_{2n}(\bR))\,|\, \rho\text{ is }(\alpha_{x_0},z)\text{-normalized}\}.
\]
For any $\rho\in \mathfrak{M}$ and $Q\in \mathfrak{a}$, we have $\widetilde{\Phi}^{Q,\alpha_{x_0}}(\rho)\in \mathfrak{M}$. Also, by Lemma~\ref{lem:normalizationmax}, 
\[
\widetilde{N}=\{\rho\in \mathfrak{M}\,|\,[\rho]\in N\}
\]
is compact. 

If there were no such a constant $C_N$, there would be a sequence $\rho_i\in \widetilde{N}$ and $Q_i\in\mathfrak{a}$ such that $\Phi^{Q_i,\alpha_{x_0}}(\rho_i)\in \widetilde{N}$ and $\|Q_i\|\to \infty$ as $i\to \infty$.  

Suppose first that $\alpha$ is non-separating. Recall  that the restriction maps $\widetilde{\mathcal{R}}_1 :\mathfrak{M} \to \Hom_{\mathrm{M}}(\pi_1(S_1) , \Sp_{2n}(\bR))$ is continuous. Moreover, the restriction map $\widetilde{\mathcal{R}}_v : \rho\mapsto\rho( \mathfrak{j}(v))$ is also a continuous map from $\widetilde{N}$ to $\Sp_{2n}(\bR)$. From the definition of $\widetilde{\Phi}^{Q_i,\alpha_{x_0}}$, we know that 
\begin{align}
\widetilde{\mathcal{R}}_1\circ \widetilde{ \Phi}^{Q_i,\alpha_{x_0}} &= \widetilde{\mathcal{R}}_1 \label{R1action}\\
\widetilde{\mathcal{R}}_v\circ \widetilde{\Phi}^{Q_i,\alpha_{x_0}}(\rho_i)&= \widetilde{R}_v(\rho_i) \exp (Q_i)\label{R2action}.
\end{align}
By (\ref{R2action}) and  the assumption $\| Q_i \| > i$, we know that 
\[
\{ Q \,|\,\widetilde{\mathcal{R}}_v (\widetilde{N})\exp Q \cap \widetilde{\mathcal{R}}_v (\widetilde{N}) \ne \emptyset\}
\]
is a unbounded subset of $\mathsf{G}$. Since $\widetilde{R}_v(\widetilde{N})$ is compact, this violates the fact that the right multiplication action on $\Sp_{2n}(\bR)$ is proper.

Now we handle the case when $\alpha$ is separating. Recall that the restriction maps $\widetilde{\mathcal{R}}_i :\widetilde{N}\to \Hom_{\mathrm{M}}(\pi_1(S_i) , \Sp_{2n}(\bR))$ are continuous. Moreover, from the definition of $\widetilde{\Phi}^{Q_i,\alpha_{x_0}}$, we know that 
\begin{align}
\widetilde{\mathcal{R}}_1\circ  \widetilde{\Phi}^{Q_i,\alpha_{x_0}} &= \widetilde{\mathcal{R}}_1 \label{R1actionsep}\\
\widetilde{\mathcal{R}}_2 \circ  \widetilde{\Phi}^{Q_i,\alpha_{x_0}}&= \Ad_{\exp Q_i}\circ \widetilde{\mathcal{R}}_2. \label{R2actionsep}
\end{align}
By (\ref{R2actionsep}) we know that
\[
\{ Q\,|\,\Ad_{\exp Q}(\widetilde{\mathcal{R}}_2 (\widetilde{N}))\cap \widetilde{\mathcal{R}}_2 (\widetilde{N}) \ne \emptyset\}
\]
is unbounded.  Since $\widetilde{\mathcal{R}}_2(\widetilde{N})$ is a compact subset of $\Hom_{\mathrm{M}} (\pi_1(S_2),\Sp_{2n}(\bR))$, the conjugation action is proper (Lemma 1 of \cite{Kimi01} when the Lie group is $\PSL_n(\bR)$ or Proposition~1.1 of \cite{JM87} for more general and abstract setting), a contradiction. This completes the proof. 
\end{proof}

\subsection{Goldman Flows}  \label{sec:gflow} 
Let $t \in \bR$ and let $Q\in \mathfrak{a}$. The \emph{Goldman flow} on $\Hit_\mathsf{G}(S)$ by  $Q$ along $\alpha$ is given by  
\[
\Phi_t ^{Q,\alpha} ([\rho]):=\Phi^{tQ,\alpha}([\rho]).
\]
We can define the Goldman flow on $\Max_{2n}^+(S)$ along $\alpha$ by $Q$ as above. Unlike the Hitchin component case, a Goldman flow may not be defined everywhere on $\Max_{2n}^+(S)$. This is mainly due to the fact that $\rho(\alpha_{x_0})$ is not always loxodromic.  

The following theorem is essentially proved by Goldman although the original statement addresses the closed surface case only.
\begin{thm}[\cite{Goldman86}] Let $\Phi^{Q,\alpha} _t$ be a Goldman flow.
\begin{itemize}
\item[(i)] If $S$ is a closed surface of genus at least 2, the flow $[\rho]\mapsto \Phi^{Q,\alpha}_t( [\rho])$ defined as above is a Hamiltonian flow on $(\operatorname{Hit}_\mathsf{G}(S),\omega_{ABG})$. It is also a Hamiltonian flow on the subset of $(\Max_{2n}^+(S),\omega_{ABG})$ where the flow is defined. 
\item[(ii)] If $S$ is a compact surface with boundary of negative Euler characteristic, the flow $[\rho]\mapsto \Phi^{Q,\alpha}_t( [\rho])$ is a Hamiltonian flow on $(\operatorname{Hit}_\mathsf{G}(S,\mathbf{B}),\omega_{GHJW})$ for a choice of $\mathbf{B}$.
\end{itemize}
In either cases, $\Phi_t ^{Q,\alpha}$ is a $\vol$-preserving flow.
\end{thm}

For a split real form $\mathsf{G}$ with the positive structure $\Theta=\Delta$, we let $Q_\mathsf{G}\in \mathfrak{a}$ be any element in the boundary of a closed Weyl chamber. For instance, we may choose 
\[
Q_\mathsf{G} =\operatorname{diag}\left(1,1,-2\right)
\]
for $\mathsf{G}= \PSL_3(\bR)$ and 
\[
Q_\mathsf{G} = \operatorname{diag}(1,0,0,\cdots,0,-1)
\]
for $\mathsf{G}= \PSL_{n+1}(\bR)$, $\PSO_{n+1,n}$, $\PSO_{n,n}$ or $\PSp_{2n}(\bR)$ with $n\ge 2$.

For maximal representations, let
\begin{equation}\label{eq:QM}
Q_M:= \begin{pmatrix} I_n & 0 \\ 0 & - I_n\end{pmatrix}.
\end{equation}
We make this choice of $Q_M$ because it commutes with any matrix of the form (\ref{eq:standard}) so that the Goldman flow $\Phi^{Q_M}_t$ is well-defined everywhere on $\Max_{2n}^+(S)$.

We are particularly interested in the Goldman flow associated with $Q_{\mathsf{G}}$ and $Q_M$ along $\alpha$. We will denote $\Phi^{Q_\mathsf{G},\alpha}_t$ and $\Phi^{Q_M,\alpha}_t$ by  $\Phi_t^\alpha$ or simply by $\Phi_t$ if $\alpha$ is understood from the context. 

\begin{defn}
    We call the flows $\Phi^{Q_{\mathsf{G}},\alpha}_t$ on $\Hit_{\mathsf{G}}(S)$ (or on $\Hit_{\mathsf{G}}(S,\mathbf{B})$, if $S$ has boundary) and  $\Phi^{Q_M,\alpha}_t$ on $\Max^+_{2n} (S)$ \emph{bulging flows}.
\end{defn}

The terminology bulging flow is originated from  \cite{Goldman2013}, where the Goldman flow $\Phi^{Q_{\mathsf{PSL}_3(\bR)}, \alpha}_t$ on $\Hit_3(S)$ is studied in the context of convex projective geometry. 

The key properties of $Q_\mathsf{G}$ and $Q_M$ are that they are members of the maximal abelian subalgebra but not contained in the interior of the positive Weyl chamber.  To find such elements, the real rank of $\mathsf{G}$ must be bigger than 2. In this sense, having  a bulging flow is a characterizing property of higher rank Lie groups.

\section{The proof of the main theorems}\label{mainsection}
This section is devoted to the proof of our main theorems.

Throughout this section, $S$ will denote a compact orientable surface of negative Euler characteristic unless otherwise stated. We say that two unoriented closed curves in $S$ are isotopic if they are ambient isotopic as subsets of $S$.  By \cite[Theorem 2.1]{epstein1966}, if two oriented simple closed curves $\alpha$, $\beta$ are isotopic as embeddings $S^1\to S$ then $|\alpha|$ and $|\beta|$ are ambient isotopic as subsets of $S$. 

After some preparations, we prove Theorems~\ref{thm:main},  \ref{thm:maximal} and \ref{thm:main2} in Section~\ref{sec:mainproof}.

\subsection{Mapping class group actions and Goldman flows}
We prove that under certain conditions, the action of a mapping class group element and Goldman flows commute.

\begin{lem}\label{lem:commute} 
Let $\phi\in \Mod(S)$ be a mapping class group element. Suppose that there is an orientation preserving homeomorphism $f:S\to S$ in the class $\phi$ such that $f(|\alpha|)$ is isotopic to $|\alpha|$. Let $Q\in \mathfrak{a}$.  If $f$ preserves the orientation of $\alpha$, then  the Goldman flow $\Phi^{Q,\alpha}_t$ on $\Hit_\mathsf{G}(S)$ commutes with the $\phi$-action. The same conclusion holds for $\Max_{2n}^{+,\alpha}(S)$.
\end{lem}
\begin{proof} 
By \cite[Theorem 2.1]{epstein1966}, we may assume that  that $f(x_0)=x_0$, and $f(|\alpha|) = |\alpha|$.   Recall that we fixed a point $z\in \partial_\infty \pi_1(S_1)$ with $z\ne \alpha_{x_0}^{\pm}$. 

We first show the lemma for Hitchin representations. Let $\rho$ be any $(\alpha_{x_0},z)$-normalized $\mathsf{G}$-Hitchin representation. Since $f$ preserves the orientation of $\alpha$, we have  $f_\sharp(\alpha_{x_0}) = \alpha_{x_0}$. It follows that, either $\alpha$  is separating or non-separating, $f_\sharp$ leaves $j_\sharp(\pi_1(S_1))$ in $\pi_1(S)$ invariant. Thus, there is a $F\in \mathsf{L}_\Theta$ such that  $\Ad_F(f_*(\rho))$ is $(\alpha_{x_0},z)$-normalized. By Lemma~\ref{lem:normalization}, we know that $F$ commutes with $\rho(\alpha_{x_0})$ and $\exp Q$. 

Assume that $\alpha$ is non-separating and $f$ preserves the orientation of $\alpha$.  Then $f_\sharp$ leaves $j_\sharp(\pi_1(S_1))\subset\pi_1(S)$ invariant. By applying an inner automorphism if necessary, we may assume that $f_\sharp ^{-1} (\mathfrak{j}(v))=g \mathfrak{j}(v)$ for some $g\in \mathfrak{j}(\pi_1(S_1))$. We choose $F\in \mathsf{L}_\Theta$ such that $\Ad_F(f_*(\rho))$ is $(\alpha_{x_0},z)$-normalized. 

For $\gamma\in \mathfrak{j}(\pi_1(S_1))$, we compute
\begin{align*}
\widetilde{\Phi}^{tQ,\alpha_{x_0}}(\Ad_F(f_*(\rho)))(\gamma)&= \Ad_F (f_*(\rho))(\gamma)\\
&= F \rho (f_\sharp^{-1}(\gamma))F^{-1}.
\end{align*}
Recall that $f_\sharp ^{-1}(\gamma)\in \mathfrak{j}(\pi_1(S_1))$. By (\ref{R1action}), we have
\[
F \rho (f_\sharp^{-1}(\gamma))F^{-1}=F\widetilde{\Phi}^{tQ,\alpha_{x_0}}(\rho)(f_\sharp^{-1}(\gamma)) F^{-1}.
\]
By the definition of the action of $f_*$, 
\begin{align*}
F\widetilde{\Phi}^{tQ,\alpha_{x_0}}(\rho)(f_\sharp^{-1}(\gamma)) F^{-1}&= F f_*(\widetilde{\Phi}^{tQ,\alpha_{x_0}}(\rho))(\gamma) F^{-1}\\
&=\Ad_F(f_*(\widetilde{\Phi}^{tQ,\alpha_{x_0}}(\rho)))(\gamma).
\end{align*}
Hence, for $\gamma\in \mathfrak{j}(\pi_1(S_1))$,
\begin{equation}\label{commutenonsep1}
\widetilde{\Phi}^{tQ,\alpha_{x_0}}(\Ad_F(f_*(\rho)))(\gamma)=\Ad_F(f_*(\widetilde{\Phi}^{tQ,\alpha_{x_0}}(\rho)))(\gamma).
\end{equation}

Let $w:=\mathfrak{j}(v)$, so that $f_\sharp^{-1}(w) = g w $ for some $g\in \mathfrak{j}(\pi_1(S_1))$. Using the fact  $F\exp(tQ)=\exp(tQ)F$, we know that
\begin{align*}
\widetilde{\Phi}^{tQ,\alpha_{x_0}}(\Ad_F(f_*(\rho)))(w)&=\Ad_F(f_*(\rho))(w) \exp (tQ)\\
&= F \rho(f  _\sharp^{-1} (w))F^{-1} \exp (tQ )\\
&= F \rho(f_\sharp^{-1} (w)) \exp (tQ ) F^{-1} \\
&= F \rho(g)\rho(w) \exp (tQ) F^{-1}.
\end{align*}
Recall that $g\in \mathfrak{j}(\pi_1(S_1))$. Due to  (\ref{R1action}) and (\ref{R2action}), we know that
\begin{align*}
F \rho(g)\rho(w) \exp (tQ ) F^{-1}&= F \widetilde{\Phi}^{tQ,\alpha_{x_0}}(\rho) (g) \widetilde{\Phi}^{tQ,\alpha_{x_0}}(\rho) (w) F^{-1} \\
&= F \widetilde{\Phi}^{tQ,\alpha_{x_0}}(\rho) (gw) F^{-1}\\
&= F \widetilde{\Phi}^{tQ,\alpha_{x_0}}(\rho) (f_{\sharp}^{-1}(w)) F^{-1}.
\end{align*}
By the definition of $f_*$,
\begin{align*}
F \widetilde{\Phi}^{tQ,\alpha_{x_0}}(\rho) (f_{\sharp}^{-1}(w)) F^{-1}&= F f_*( \widetilde{\Phi}^{tQ,\alpha_{x_0}}(\rho))(w) F^{-1} \\
&=\Ad_F(f_*(\widetilde{\Phi}^{tQ,\alpha_{x_0}}(\rho)))(w).
\end{align*}
This shows that
\begin{equation}\label{commutenonsep2}
\widetilde{\Phi}^{tQ,\alpha_{x_0}}(\Ad_F(f_*(\rho)))(w)=\Ad_F(f_*(\widetilde{\Phi}^{tQ,\alpha_{x_0}}(\rho)))(w).
\end{equation}
By (\ref{commutenonsep1}) and (\ref{commutenonsep2}), it follows that 
\[
\Ad_F(f_*(\widetilde{\Phi}^{tQ,\alpha_{x_0}}(\rho)))=\widetilde{\Phi}^{tQ,\alpha_{x_0}} (\Ad_F(f_*(\rho))).
\]
On the level of $\Hit_\mathsf{G}(S)$, we obtain $ \phi_* ( \Phi^{Q,\alpha} _t ([\rho]))=\Phi^{Q,\alpha}_t( \phi_*([\rho]))$.

Suppose that $\alpha$ is separating and $f$ preserves the orientation of $\alpha$. Then we may assume that $f$ preserves each $S_i$. Let $\rho$ be any $(\alpha_{x_0},z)$-normalized representation. As we mentioned in the beginning of the proof, we choose $F\in \mathsf{L}_\Theta$ such that  $\Ad_F(f_*(\rho))$ is $(\alpha_{x_0},z)$-normalized. 

For $\gamma\in \mathfrak{j}(\pi_1(S_1))$,
\begin{align*}
\widetilde{\Phi}^{tQ, \alpha_{x_0}}(\Ad_F(f_*(\rho))) (\gamma)&= \Ad_F(f_*(\rho))(\gamma)\\
&= F f_*(\rho)(\gamma) F^{-1}\\
&=F \rho(f_\sharp^{-1}(\gamma)) F^{-1}.
\end{align*}
Since $f_\sharp^{-1}(\gamma)\in \mathfrak{j}(\pi_1(S_1))$, we use (\ref{R1actionsep}) to compute
\begin{align*}
F \rho(f_\sharp^{-1}(\gamma)) F^{-1}&=F \widetilde{\Phi}^{tQ,\alpha_{x_0}}(\rho) (f_\sharp^{-1}(\gamma)) F^{-1}\\
&= \Ad_F( f_*(\widetilde{\Phi}^{tQ,\alpha_{x_0}}(\rho))) (\gamma).
\end{align*}
Hence, for $\gamma\in \mathfrak{j}(\pi_1(S_1))$ we have
\begin{equation}\label{commsep1}
\widetilde{\Phi}^{tQ, \alpha_{x_0}}(\Ad_F(f_*(\rho))) (\gamma)=\Ad_F( f_*(\widetilde{\Phi}^{tQ,\alpha_{x_0}}(\rho))) (\gamma).
\end{equation}
Similarly, for $\gamma\in \mathfrak{j}(\pi_1(S_2))$, 
\begin{align*}
\widetilde{\Phi}^{tQ,\alpha_{x_0}}(\Ad_F(f_*(\rho)))(\gamma)&= \exp (tQ) \Ad_F(f_*(\rho))(\gamma) \exp(-tQ)\\
&= \exp(tQ) F f_*(\rho)(\gamma)F^{-1} \exp(-tQ)\\
&=F\exp(tQ) f_*(\rho)(\gamma) \exp (-tQ) F^{-1}\\
&=F \exp(tQ) \rho(f_\sharp^{-1}(\gamma)) \exp(-tQ) F^{-1}
\end{align*}
due to $F\exp(tQ) = \exp(tQ) F$. Recall that $f_\sharp^{-1}(\gamma) \in \mathfrak{j}(\pi_1(S_2))$. By (\ref{R2actionsep}), we have
\begin{align*}
F \exp(tQ) \rho(f_\sharp^{-1}(\gamma)) \exp(-tQ) F^{-1}&=F\widetilde{\Phi}^{tQ,\alpha_{x_0}}(\rho)(f_\sharp^{-1}(\gamma))F^{-1}\\
&= \Ad_F (f_*(\widetilde{\Phi}^{tQ,\alpha_{x_0}}(\rho)))(\gamma).
\end{align*}
This yields
\begin{equation}\label{commsep2}
  \widetilde{\Phi}^{tQ,\alpha_{x_0}}(\Ad_F(f_*(\rho)))(\gamma) =  \Ad_F (f_*(\widetilde{\Phi}^{tQ,\alpha_{x_0}}(\rho)))(\gamma)
\end{equation}
for $\gamma\in \mathfrak{j}(\pi_1(S_2))$. Therefore, combining (\ref{commsep1}) and  (\ref{commsep2}), we know that 
\[
\Ad_F(f_*(\widetilde{\Phi}^{tQ,\alpha_{x_0}}(\rho)))=\widetilde{\Phi}^{tQ,\alpha_{x_0}} (\Ad_F(f_*(\rho))).
\]
It follows that $ \phi_* ( \Phi^{Q,\alpha} _t ([\rho]))=\Phi^{Q,\alpha}_t( \phi_*([\rho]))$.

Now we handle the maximal representations. Let $\rho$ be a  $(\alpha_{x_0},z)$-normalized maximal representation such that $[\rho]\in \Max_{2n}^{+,\alpha}(S)$.  Since $\rho(\alpha_{x_0})$ is loxodromic and $f_*(\rho)(\alpha_{x_0})=\rho(\alpha_{x_0})$, we use Lemma~\ref{lem:normalizationmax} to find a diagonal matrix $F$ such that  $\Ad_F(f_*(\rho))$ is $(\alpha_{x_0},z)$-normalized. Since $F$ commutes with $\rho(\alpha_{x_0})$ and  $\exp Q$, the above proof for Hitchin representations works equally for maximal representations.
\end{proof}

\subsection{Finding an initial open set}\label{findopenset}
In this subsection, we will construct a nice initial open set. Our construction heavily relies on the collar lemmas by Beyrer--Guichard--Labourie--Pozzetti--Wienhard and Burger-Pozzetti. For $\PSL_n(\bR)$-Hitchin representations, Lee--Zhang established the first explicit collar lemmas; see \cite{LZ17}. 

We first focus on the Hitchin case, where $\mathsf{G}$ is a split real form with the positive structure $\Theta= \Delta$. Let $\lambda:\mathsf{G} \to \overline{\mathfrak{a}^+}$ be the Jordan projection. Let $\theta$ be a linear functional on $\mathfrak{a}$. For an oriented essential closed curve $\beta$, and a conjugacy class $[\rho]$ of a $\Theta$-positive representation $\rho:\pi_1(S)\to \mathsf{G}$, we denote the $\theta$-length (or the $\theta$-period) of $\beta$ by
\[
l ^\theta _{[\rho]}(\beta) := \theta \circ \lambda(\rho(\beta'))+\theta \circ \lambda(\rho(\beta')^{-1})
\]
where $\beta'$ is a loop based at $x_0$ freely homotopic to $\beta$. Note that $l^\theta _{[\rho]}$ does not depend on the choice of a representation $\rho$ in the conjugacy class $[\rho]$ nor the choice of $\beta'$ in the free homotopy class of $\beta$. For this reason, we frequently write $l^\theta _\rho$ instead of $l^\theta _{[\rho]}$. Moreover, observe that $l^\theta _{\rho}(\beta) = l^\theta _{\rho}(\overline{\beta})=l^\theta _{\rho}(|\beta|)$, where $\overline{\beta}$ means the orientation reversal. Therefore, $l^\theta _{\rho}$ is indeed functions on the isotopy classes of unoriented essential closed curves.

Let $\omega_{\theta_1}$ be the fundamental weight attached to the root $\theta_1\in \Delta$. As we will only use the $\omega_{\theta_1}$-length in this paper, we write $l_\rho = l_\rho ^{\omega_{\theta_1}}$ for simplicity. 

Now we state our collar lemma which is essentially due to Beyrer--Guichard--Labourie--Pozzetti--Wienhard \cite{beyrer2024}.

\begin{thm}\label{thm:collar} Let $\mathsf{G}$ be a split real form with the positive structure $\Theta=\Delta$. Let $\rho:\pi_1(S)\to \mathsf{G}$ be a $\Theta$-positive representation. Let $\alpha$, $\beta$ be two essential closed curves having non-trivial geometric intersection number. Then we have
    \[
    \exp(-l_\rho (\alpha)) + \exp (-l_\rho(\beta)) < 1.
    \]
\end{thm}
\begin{proof}
    By \cite[Corollary D]{beyrer2024}, we have
    \begin{equation}\label{eq:GWcollar}
    \exp (-l_\rho(\alpha))+\exp (-\theta_1 \circ \lambda(\rho(\beta)) )<1.
    \end{equation}
    We claim that 
    \[
    l_\rho(\beta)\ge \theta_1\circ \lambda(\rho(\beta)).
    \]
    
    Let $w_0$ be the longest element of the Weyl group. Since $\lambda(x^{-1}) = -w_0 \cdot\lambda(x)$ for any $x\in \mathsf{G}$, we know that 
    \[
    l_\rho (\beta) = \omega_{\theta_1}(\rho(\beta)) + \omega_{-w_0 \cdot \theta_1}(\rho(\beta)).
    \]
    According to Table~\ref{tab:fundamentalweights} in Section~\ref{sec:notation}, $\omega_{\theta_1}+\omega_{-w_0\cdot \theta_1}$ can be written as a positive linear combination of simple roots and the coefficient of $\theta_1$ in this expression is at least 1. This shows that
    \[
    \omega_{\theta_1} + \omega_{-w_0\cdot \theta_1}\ge  \theta_1.
    \]
    Therefore, we know that
    \begin{equation}\label{eq:weightbound}
    \exp (-\theta_1 \circ \lambda(\rho(\beta))) \ge \exp (- l_\rho(\beta)).
    \end{equation}
    We obtain the desired inequality by combining (\ref{eq:GWcollar}) and (\ref{eq:weightbound}). 
\end{proof}

Now we can present our main construction.

\begin{prop} \label{prop:length}  
Let $S$ be a closed orientable surface of genus $g>1$. Let $\alpha$ be a given essential simple closed curve in $S$. Let $\{|\alpha|=\alpha_1,\alpha_2, \cdots, \alpha_{3g-3}\}$ be unoriented simple closed curves giving us a pair-of-pant decomposition of $S$. Then there exist a constant $\epsilon>0$, a nonempty  pre-compact open set $N$ of $\Hit_{\mathsf{G}}(S)$ satisfying the following properties: 
\begin{itemize}
\item[(i)] For any $t,s\in \bR$, $[\rho_1], [\rho_2]\in N$ and any $i=1,2,\cdots,3g-3$, we have 
\[
\left|l_{\Phi_t([\rho_1])}(\alpha_i)-l_{\Phi_s([\rho_2])}(\gamma)\right|>\epsilon,
\]
unless $\gamma$ is isotopic to $\alpha_i$.
\item[(ii)] $\{\phi\in \Mod(S)\,|\,\phi_*(N)\cap N \ne \emptyset\}=\{1\}$.  In particular $N$ avoids orbifold points of $\mathcal{M}_{\mathsf{G}}(S)$.
\item[(iii)] If $\alpha$ is non-separating, $S\setminus\nu(|\alpha|)=S_1$, 
\[
\{\phi\in \Mod^*(S_1,\partial S_1)\,|\, \phi_*(\mathcal{R}_1(N))\cap \mathcal{R}_1(N)\ne\emptyset\} =\{1\};
\]
if $\alpha$ is separating, $S\setminus\nu(|\alpha|)=S_1\cup S_2$,
\[
\{\phi\in \Mod(S_i)\,|\,\phi_*(\mathcal{R}_i(N))\cap \mathcal{R}_i(N)\ne \emptyset\}=\{1\} \text{ for }i=1,2
\]
where $\mathcal{R}_i$ are the restriction maps defined in Section~\ref{sec:decomposition}.
\end{itemize}
\end{prop}
\begin{proof} 
Let $D_{\mathsf{G}}$ be the unique real solution for the equation $2e^{-x} = 1$ where we have $0 < D_{\mathsf{G}}< 1$. For any $\mathsf{G}$-Hitchin representation $\rho$ and any two essential closed curves $\alpha$, $\gamma$ with $i(\alpha, \gamma)\ne 0$,  $l_{\rho}(\alpha) < D_{\mathsf{G}}$ implies $l_{\rho}(\gamma) > D_{\mathsf{G}}$ by Theorem~\ref{thm:collar}.

Let $\{E,F,H\}$ be the regular $\mathfrak{sl}_2$-triple defining the representation $\tau_\mathsf{G}:\PSL_2(\bR) \to \mathsf{G}$ and let $\theta_1\in \Delta$ be the simple root. Define 
\[
d_\mathsf{G} := \omega_{\theta_1}(H)+\omega_{-w_0\cdot \theta_1}(H).
\]
We construct a hyperbolic structure on $X$ on $S$  using the Fenchel-Nielsen coordinates such that the hyperbolic lengths $l_X$ of $\alpha_1,\cdots,\alpha_{3g-3}$ are all distinct and smaller than $\frac{2 D_{\mathsf{G}}}{ d_{\mathsf{G}}}$. Let $\rho_F:\pi_1(S)\to \PSL_{2}(\bR)$ be the holonomy of this hyperbolic structure.  Since 
\[
\lambda(\tau_{\mathsf{G}}\circ \rho_F(\alpha_i )) = \frac{l_{X}(\alpha_i)}{2}H,
\]
we know that the Fuchsian representation $\rho_0 = \tau_{\mathsf{G}}\circ \rho_F$ satisfies 
\[
l_{\rho_0}(\alpha_i)=\frac{l_X(\alpha_i)}{2}\cdot d_{\mathsf{G}}<D_{\mathsf{G}}
\]
for all $i=1,2,\cdots, 3g-3$ and that 
\[
l_{\rho_0}(\alpha_i)\ne l_{\rho_0}(\alpha_j)
\]
for all $i\ne j$. 

Observe that the map $L:\Hit_{\mathsf{G}}(S)\to \bR^{3g-3}$ defined by 
\[
[\rho]\mapsto (l_{\rho}(\alpha_1),\, l_{\rho}(\alpha_2),\, \cdots,\, l_{\rho}(\alpha_{3g-3}))
\]
is continuous. Let 
\[
\epsilon=\frac{1}{3}\min\left\{\min_{i\ne j} |l_{\rho'}(\alpha_i) - l_{\rho'}(\alpha_j)|,\,\min_k |l_{\rho'}(\alpha_k)-D_{\mathsf{G}}|\right\}>0
\]
and define $N_1$ to be the inverse image of the open neighborhood
\[
\prod_{i=1} ^{3g-3}(l_{\rho'}(\alpha_i)-\epsilon,\,l_{\rho'}(\alpha_i)+\epsilon)
\]
by the continuous map $L$. Then for any $[\rho_1]$ and $[\rho_2]$ in $N_1$,
\begin{itemize}
    \item[(1)] $l_{\rho_1}(\alpha_i)<D_{\mathsf{G}}-\epsilon$ for all $i=1,2,\cdots,3g-3$
    \item[(2)] $|l_{\rho_1}(\alpha_i) - l_{\rho_2}(\alpha_j)|>\epsilon$ for all $i\ne j$.
\end{itemize}

We claim that the set $N_1$ satisfies (i). For this, choose any $[\rho_1]$ and $[\rho_2]$ in $N_1$ and any real numbers $s$ and $t$. Let $\gamma$ be an essential closed curve not isotopic to $\alpha_i$. Since $\{\alpha_1,\cdots, \alpha_{3g-3}\}$ is a maximal collection of disjoint essential simple closed curves,  $\gamma$ is either isotopic to one of $\alpha_j$ or intersects essentially with some $\alpha_j$. If $\gamma$ is isotopic to $\alpha_j$ for some $j\ne i$, we have, by property (2),
\[
|l_{\Phi_t([\rho_1])}(\alpha_i) - l_{\Phi_s([\rho_2])}(\gamma)|=|l_{\rho_1}(\alpha_i) - l_{\rho_2}(\alpha_j)|>\epsilon.
\]
Suppose that $\gamma$ intersects essentially with some $\alpha_j$. We know that 
\[
l_{\Phi_s([\rho_2])}(\alpha_j)=l_{\rho_2}(\alpha_j)<D_{\mathsf{G}}-\epsilon
\]
by property (1). Then by our choice of the constant $D_{\mathsf{G}}$, 
\[
l_{\Phi_s([\rho_2])}(\gamma)>D_{\mathsf{G}}.
\]
Therefore,
\[
|l_{\Phi_t([\rho_1])}(\alpha_i) - l_{\Phi_s([\rho_2])}(\gamma)|=|l_{[\rho_1]}(\alpha_i) - l_{\Phi_s([\rho_2])}(\gamma)|>\epsilon
\]
by the property (1) again. This shows that the set  $N_1$ has property (i).

By Lemma~\ref{emptyinterior}, we may find a pre-compact non-empty open $N_2\subset N_1$ such that (ii) holds.

Now assume that $\alpha$ is non-separating.  Choose any $[\rho]\in N_2$ and let 
\[
\mathbf{B}_1 := ([\rho(\alpha_{x_0})],[\rho(\alpha_{x_0})]).
\]
Then 
\[
\mathcal{R}_1(N_2)\cap \Hit_{\mathsf{G}}(S_1,\mathbf{B}_1)
\]
is a non-empty open set. Since $\Mod^*(S_1,\partial S)$ acts effectively and properly on $\Hit_{\mathsf{G}}(S_1,\mathbf{B}_1)$ we apply Lemma~\ref{emptyinterior} to find an element 
\[
[\rho']\in \mathcal{R}_1(N_2)\cap \Hit_{\mathsf{G}}(S_1, \mathbf{B}_1)
\]
with the trivial stabilizer. Since $\Mod^*(S,\partial S)$ also acts properly and effectively on $\Hit_{\mathsf{G}}(S_1)$, we can find a  open neighborhood $U\subset \Hit_{\mathsf{G}}(S_1)$ of $[\rho']$ such that 
\[
\{\phi\in \Mod^*(S_1,\partial S)\,|\,\phi_*(U) \cap U\ne\emptyset\}=\{1\}.
\]
Then $N= \mathcal{R}_1^{-1} (U)\cap N_2$ the desired open set satisfying (i) -- (iii). 

Now assume that $\alpha$ is separating. Choose $[\rho]\in N_2$ and set $[\rho_i] = \mathcal{R}_i([\rho])$ for $i=1,2$. Let $\zeta^1$ and $\zeta^2$ be oriented boundary components as a result of cutting $S$ along $\alpha$ (see Section~\ref{fungpdecomsep}). Let $\mathbf{B}_1=\mathbf{B}_2=([\rho(\alpha_{x_0})])$. We know that 
\[
[\rho_i]\in \mathcal{R}_i(N_2)\cap \Hit_{\mathsf{G}}(S_i,\mathbf{B}_i).
\]
Hence, 
\[
\mathcal{R}_i(N_2)\cap \Hit_{\mathsf{G}}(S_i,\mathbf{B}_i)
\]
is non-empty and open in $\Hit_{\mathsf{G}}(S_i,\mathbf{B}_i)$. Since $\Mod^*(S_i, \partial S_i)$ acts properly and effectively on $\Hit_{\mathsf{G}}(S_i, \mathbf{B}_i)$,  Lemma~\ref{emptyinterior} shows that there is 
\[
[\rho_i']\in \mathcal{R}_i(N_2)\cap \Hit_n(S_i, \mathbf{B}_i)
\]
having the trivial stabilizer in $\Mod^*(S_i, \partial S_i)$.  As $\Mod^*(S_i, \partial S_i)$ acts properly and effectively on $\Hit_{\mathsf{G}}(S_i)$, there is an open neighborhood $U_i\subset \Hit_{\mathsf{G}}(S_i)$ of $[\rho'_i]$ such that  
\[
\{\phi\in \Mod^*(S_i,\partial S_i)\,|\,\phi_*(U_1)\cap U_1\ne\emptyset\}=\{1\}.
\]
We know that 
\[
U_1\times U_2 \cap \mathcal{R}_1\times\mathcal{R}_2 (N_2) \ne \emptyset.
\]
Hence, 
\[
N=\mathcal{R}_1^{-1}(U_1) \cap \mathcal{R}_2 ^{-1} (U_2) \cap N_2
\]
is non-empty and satisfies (i) -- (iii).
\end{proof} 

Of course you can do any deformation by Goldman flows for $\alpha$ and the conclusion of Proposition~\ref{prop:length} still holds.

We can also show the relative version of Proposition~\ref{prop:length}. We only state the proposition and omit the proof because it is almost the same to the proof of Proposition~\ref{prop:length}. Here we need Lemma~\ref{boundaryfreedom} to ensure that our construction results in a non-empty set for any choice of the boundary holonomy.

\begin{prop} \label{prop:lengthrel}  
Let $S$ be a compact orientable surface of negative Euler characteristic with genus $g$ and $b>0$ boundary components. Assume $(g,b)\ne(0,3)$. Let $\alpha$ be a given essential simple closed curve in $S$. Let $\{|\alpha|=\alpha_1,\alpha_2, \cdots, \alpha_{3g+b-3}\}$ be unoriented simple closed curves giving us a pair-of-pant decomposition of $S$. Let $\mathbf{B}$ be a choice of boundary holonomy. Then there exist a constant $\epsilon>0$, a nonempty pre-compact open set $N$ of  $\Hit_n(S,\mathbf{B})$ satisfying the following properties: 
\begin{itemize}
\item[(i)] For any $t,s\in \bR$, $[\rho_1], [\rho_2]\in N$ and any $i=1,2,\cdots,3g+b-3$, we have 
\[
\left|l_{\Phi_t([\rho_1])}(\alpha_i)-l_{\Phi_s([\rho_2])}(\gamma)\right|>\epsilon,
\]
unless $\gamma$ is isotopic to $\alpha_i$.
\item[(ii)] $\{\phi\in \Mod(S)\,|\,\phi_*(N)\cap N \ne \emptyset\}=\{1\}$.  In particular $N$ avoids orbifold points of   $\mathcal{M}_n(S,\mathbf{B})$.
\item[(iii)] If $\alpha$ is non-separating, $S\setminus\nu(|\alpha|)=S_1$, 
\[
\{\phi\in \Mod^*(S_1,\partial S)\,|\, \phi_*(\mathcal{R}_1(N))\cap \mathcal{R}_1(N)\ne\emptyset\} =\{1\};
\]
if $\alpha$ is separating, $S\setminus\nu(|\alpha|)=S_1\cup S_2$,
\[
\{\phi\in \Mod(S_i)\,|\,\phi_*(\mathcal{R}_i(N))\cap \mathcal{R}_i(N)\ne \emptyset\}=\{1\} \text{ for }i=1,2
\]
where $\mathcal{R}_i$ are the restriction maps defined in Section~\ref{sec:decomposition}.
\end{itemize}
\end{prop}

Now we state the maximal component version of Proposition~\ref{prop:length}. Given a maximal representation $[\rho]\in \Max_{2n}(S)$, we define the length $L_{[\rho]}(\beta)$ of an essential closed curve $\beta$ by
\[
L_{[\rho]} (\beta) := 2 \sqrt{\sum_{i=1} ^ n \log^2 |\lambda_i(\rho(\beta'))|}.
\]
We also recall the following version of collar lemma for maximal representations due to Burger--Pozzetti.
\begin{thm}[\cite{burger2017}]\label{thm:collarmax}
Let $S$ be a closed orientable surface of genus $>1$ and let $[\rho]\in\Max_{2n}(S)$. For any closed curve $\gamma$ meeting $\alpha$ essentially, we have
\[
\left(\exp{\frac{L_\rho(\alpha)}{\sqrt{n}}}-1\right)\left(\exp\frac{L_\rho(\gamma)}{\sqrt{n}}-1\right)\ge 1.
\]
\end{thm}

\begin{prop} \label{prop:lengthmax} 
Let $S$ be a closed orientable surface of genus $>1$. Let $\alpha$ be a non-separating essential simple closed curve in $S$. Let $\{|\alpha|=\alpha_1,\alpha_2, \cdots, \alpha_{3g-3}\}$ be unoriented simple closed curves giving us a pair-of-pant decomposition of $S$.   Let $C$ be a connected component of $\Max_{2n}(S)$. Then there exist a constant $\epsilon>0$, a nonempty pre-compact open set $N$ of $C\cap \Max_{2n}^+(S)$ having the following properties:
\begin{itemize}
\item[(i)]  For any $t,s\in \bR$, $[\rho_1], [\rho_2]\in N$ and any $i=1,2,\cdots,3g-3$, we have 
\[
\left|L_{\Phi_t([\rho_1])}(\alpha_i)-L_{\Phi_s([\rho_2])}(\gamma)\right|>\epsilon,
\]
unless $\gamma$ is isotopic to $\alpha_i$.
\item[(ii)] $\rho(\alpha_{x_0})$ is loxodromic for all $[\rho]\in N$. 

\item[(iii)] $\{\phi\in \Mod(S)\,|\,\phi_*(N)\cap N\ne \emptyset\} =\{1\}$. Hence $N$ avoids orbifold points of $\mathcal{M}_n(S)$.
\item[(iv)] $\{\phi\in \Mod^*(S_1)\,|\,\phi_*(\mathcal{R}_1(N))\cap \mathcal{R}_1(N)\ne \emptyset\}=\{1\}$, where $\mathcal{R}_0$ is the restriction map defined in Section~\ref{sec:decomposition}.
\end{itemize}
\end{prop}

To show Proposition~\ref{prop:lengthmax}, we need some preparation. 

Let  $\rho\in \Hom(\pi_1(S),\SL_2(\bR))$ be a lift of the holonomy of a hyperbolic structure. Recall that we defined the length of an essential closed curve $\gamma$ by 
\[
l_\rho (\gamma) = \log \left|\frac{\lambda_1(\rho(\gamma))}{\lambda_2(\rho(\gamma))}\right|=2\log |\lambda_1(\rho(\gamma))|.
\]
We observe that for the irreducible maximal representation, we have 
\[
L_{\tau_{2n}\circ \rho}= \sqrt{\frac{n(4n^2-1)}{3}}\cdot l_{\rho}.
\]
For the diagonal representation, we have 
\[
L_{\delta_{2n}\circ(I_n\otimes \rho)} = \sqrt{n} \cdot l_\rho.
\]
Given these observations, we prove the following:

\begin{lem}\label{lem:short}
    Let $S$ be an orientable closed surface of genus $>1$ and let $\{\alpha_1= |\alpha|, \alpha_2, \cdots, \alpha_{3g-3}\}$ be a pants-decomposition of $S$ extending $\alpha$. Let $C$ be a connected component of $\Max_{2n}(S)$.  For any given $\epsilon>0$, one can find a maximal representation $[\rho]\in C\cap \Max^+ _{2n}(S)$ such that $L_\rho(\alpha_i) < \epsilon$ for all $i$, $L_{\rho}(\alpha_i)\ne L_{\rho}(\alpha_j)$ for all $i\ne j$, and $\rho(\alpha)$ has only real eigenvalues.
\end{lem}
\begin{proof}
First assume that $n>2$. In this case, we furthermore prove that $C$ contains a copy of the Teichm\"uller space. Then, by using the Fenchel-Nielsen coordinates and Lemma~\ref{lem:generic}, the lemma follows. 

Among $3\cdot  2^{2g}$ connected components of $\Max_{2n}(S)$, there are $2^{2g}$ components obtained by lifting the $\PSp_{2n}(\bR)$-Hitchin component. As we already know that the $\PSp_{2n}(\bR)$-Hitchin component contains a copy of $\mathcal{T}(S)$, we may assume that $C$ is not a lift of the Hitchin component. In this case, Theorem~\ref{thm:ComponentContainsFuchsian} shows that $C$ contains a twisted diagonal representation. Hence, it suffices to show that $[\rho]\mapsto [ \delta_{2n} \circ(\kappa\otimes  \rho)]$ is injective, where $\kappa:\pi_1(S)\to \mathsf{O}_n$ is a representation for the twisting corresponding to the connected component $C$. Let $[\rho_1]\ne[\rho_2]$ be elements in $\mathcal{T}(S)$. One can find an essential simple closed curve $\gamma$ such that $l_{\rho_1}(\gamma)\ne l_{\rho_2}(\gamma)$. Since 
\[
L_{ \delta_{2n}\circ(\kappa\otimes \rho)}(\gamma)=\sqrt{n}\cdot l_\rho(\gamma) L_{\kappa\otimes I}(\gamma),
\]
we have 
\[
L_{\delta_{2n} \circ(\kappa\otimes \rho_1)}(\gamma)\ne L_{\delta_{2n} \circ (\kappa\otimes \rho_2)}(\gamma).
\]
This shows that $[\delta_{2n}\circ(\kappa\otimes \rho_1)] \ne [\delta_{2n}  \circ(\kappa\otimes \rho_2)]$. Finally, as we can assume that $\kappa$ is a homomorphism into  $\mathbb{Z}_2$ generated by $\operatorname{diag}(-1,1,1,\cdots,1)$, eigenvalues of $\delta_{2n}\circ(\kappa\otimes \rho)(\alpha)$ are all real numbers. 

Now suppose that $n=2$. Again when $C$ is a one of lifted Hitchin components, the result readily follows. Thus, we assume that $C$ is not such components. Then, by Theorem~\ref{thm:ComponentContainsFuchsian}, one finds a hybrid maximal representation in $C$ by amalgamating irreducible and diagonal Fuchsian representations along a separating essential simple closed curve. Hence, combining the Fenchel-Nielsen coordinates on each subsurface and Lemma~\ref{lem:generic}, we obtained a desired maximal representation. 
\end{proof}

\begin{proof}[Proof of Proposition~\ref{prop:lengthmax}]
    Let $D^M_{2n} = \sqrt{n} \log 2$. By Theorem~\ref{thm:collarmax}, if $L_\rho(\alpha)< D_{2n} ^M $ then  $L_{\rho}(\gamma)>D_{2n} ^M$ for any $\gamma$ intersecting with $\alpha$ essentially.

    Lemma~\ref{lem:short}, Lemma~\ref{lem:simpleeigenvalues} and the continuity of the length function guarantee that there is  $[\rho_0]\in C\cap \Max_{2n}^+(S)$ such that
    \begin{itemize}
        \item $L_{\rho_0} (\alpha_i) < D^M_ {2n}$ for all $i=1,2,\cdots, 3g-3$
        \item $L_{\rho_0} (\alpha_i ) \ne L_{\rho_0}(\alpha_j)$ for all $i\ne j$ and
        \item $\rho_0(\alpha_{x_0})$ is loxodromic. 
    \end{itemize}
    Since these conditions are open conditions, we use the same argument as in the proof of Proposition~\ref{prop:length} to  find an open neighborhood $N_1\subset C \cap \Max_{2n}^+(S)$ of $[\rho_0]$ and a positive real number $\epsilon$ such that
        \begin{itemize}
        \item[(1)] $L_{\rho} (\alpha_i) < D^M_ {2n}-\epsilon$ for all $i=1,2,\cdots, 3g-3$ and all $[\rho]\in N_1$
        \item[(2)] $|L_{\rho_1}  (\alpha_i ) - L_{\rho_2}(\alpha_j)|>\epsilon$ for all $i\ne j$ and all $[\rho_1],[\rho_2]\in N_1$ and
        \item[(3)] $\rho(\alpha_{x_0})$ is loxodromic for all $[\rho]\in N_1$. 
    \end{itemize}
    In particular, every element in $N_1$ satisfies (ii). 
    
    To show that $N_1$ has property (i), we choose any $[\rho_1]$ and $[\rho_2]$ in $N_1$ and any real numbers $s$ and $t$. Suppose that we are given an essential simple closed curve $\gamma$ which is not isotopic to $\alpha_i$. If $\gamma$ is isotopic to $\alpha_j$ for some  $j\ne i$, we have 
    \[
    |L_{\Phi_t([\rho_1])}(\alpha_i) -   L_{\Phi_s([\rho_2])}(\gamma)|=|    L_{[\rho_1]}(\alpha_i) -   L_{[\rho_2]}(\alpha_j)| >\epsilon
    \]
    by property (2). If $\gamma$ is not isotopic to any of $\alpha_j$, then $\gamma$ intersects some $\alpha_j$ essentially. By property (1) we have 
    \[
    L_{\Phi_s([\rho_2])}(\alpha_j)=L_{\rho_2}(\alpha_j)<D^M_{2n}-\epsilon.
    \]
    This forces $L_{\Phi_s([\rho_2])}(\gamma)>D^M_{2n}$ by our choice of the constant $D^M_{2n}$. Due to property (1) again, it follows that
    \[
    |L_{\Phi_t([\rho_1])}(\alpha_i) -   L_{\Phi_s([\rho_2])}(\gamma)| =|    L_{[\rho_1]}(\alpha_i) -   L_{\Phi_s([\rho_2])}(\gamma)| >\epsilon.
    \]
    This shows that $N_1$ satisfies (i).
    
    Since the stabilizer of  $C$ in $\Mod(S)$ acts effectively on $C$,  Lemma~\ref{emptyinterior} allows us to find a subset $N_2\subset N_1$ such that $N_2$ is pre-compact and enjoys property (iii). To obtain $N\subset N_2$ with the condition (iv), we just follow the argument in the proof of Proposition~\ref{prop:length}.
\end{proof}

\subsection{Proofs of Theorems~\ref{thm:main}, \ref{thm:maximal} and \ref{thm:main2}}\label{sec:mainproof}

\begin{lem} \label{lem:mono}
Let $\rho$ and $\rho'$ be two $\mathsf{G}$-Hitchin representations (maximal representations). 
Suppose that for all unoriented essential closed curve $\beta$ on $S$ which is not isotopic to $|\alpha|$ we have $l_{\rho'}(\alpha) \ne l_\rho (\beta)$ ($L_{\rho'}(\alpha)\ne L_\rho (\beta)$ respectively).  Then for any homeomorphism $f: S \ra S$ where $f_\ast([\rho]) = [\rho']$, $f(|\alpha|)$ must be isotopic to $|\alpha|$. 
\end{lem} 
\begin{proof} 
Since $[\rho']= f_*([\rho]) = [\rho\circ (f_\sharp) ^{-1}]$, we know that $l_{\rho'}(\alpha)=l_\rho (f^{-1}(|\alpha|))$. If $f^{-1}(|\alpha|)$ were not isotopic to $|\alpha|$, we would have $l_\rho(f^{-1}(\alpha))\ne l_{\rho'}(\alpha)$ from the assumption. This is absurd, showing that $|\alpha|$ and $f^{-1}(|\alpha|)$ are isotopic. That is to say, $|\alpha|$ and $f(|\alpha|)$ are isotopic. The same proof works for the maximal representations.
\end{proof}

\begin{lem} \label{lem:disj} 
Let $S$ be a compact orientable surface of negative Euler characteristic possible with boundary. Assume that $S$ is not a sphere with three boundary components. Let $N$ be chosen as in Proposition~\ref{prop:length} (or Proposition~\ref{prop:lengthrel} if $S$ has boundary). Let $\Pi$ be either $\Hit_{\mathsf{G}}(S)\to \mathcal{M}_{\mathsf{G}}(S)$ (or $\Hit_n(S,\mathbf{B}) \to \mathcal{M}_n(S, \mathbf{B})$ respectively) as in Section~\ref{sec:MODaction}. Then there exists $D_N > 0$ with the following properties:
\begin{itemize} 
\item[(i)] $\Pi(N) \cap \Pi( \Phi_t(N))=\emp$ provided $t > D_N$. 
\item[(ii)] $\Pi| \Phi_t(N)$ is injective  for every $t$.
\item[(iii)] $\Pi( \Phi_t(N) )\cap \Pi( \Phi_s(N))=\emp$ provided $|s-t|> D_N$. 
\end{itemize}
The same assertion holds for a closed orientable surface $S$, the open set $N$ from Proposition~\ref{prop:lengthmax} and the projection $\Pi:\Max_{2n}(S)\to\Max_{2n}(S)/\Mod(S)$. 
\end{lem} 
\begin{proof} 
We will study the Hitchin components only. We can prove for the maximal components by replacing $Q_\mathsf{G}$ with $Q_M$. 

(i) Suppose that (i) does not hold. Then there is a divergent sequence $t_1<t_2<\cdots$ of real numbers and mapping class group elements $\phi_j\in \Mod(S)$ such that $(\phi_j)_*\Phi_{t_j}([\rho_j])= [\rho_j']$ for $[\rho_j]$, $[\rho_j']\in N$. 
Let $f_j:S\to S$ be a representative of $\phi_j$. By Proposition~\ref{prop:length} and Lemma~\ref{lem:mono}, for each $j$, we may assume without loss of generality that $f_j(|\alpha|)$ is isotopic to $|\alpha|$. Hence, $f_j$ can be isotoped so that $f_j(\nu(|\alpha|))=\nu(|\alpha|)$ for some tubular neighborhood $\nu(|\alpha|)$ of $|\alpha|$. 

We now claim that each $f_j$ is a power of the Dehn twist along $\alpha$.

First, assume that  $\alpha$ is non-separating. As in Section~\ref{fungpdecomnonsep}, let $S_1=S\setminus\nu(|\alpha|)$ and let $\mathcal{R}_1:\Hit_\mathsf{G}(S,\mathbf{B})\to\Delta_\alpha(\mathbf{B})$ be the restriction map. Let $f^1_j:S_1\to S_1$ be the restricted homeomorphism $f^1_j:=f_j|S_1$. Observe that  
\[
\mathcal{R}_1 ((f_j)_*(\Phi_{t_j} ([\rho_j]))) = (f_j ^1)_* (\mathcal{R}_1(\Phi_{t_j}([\rho_j]))) = \mathcal{R}_1([\rho_j'])\in \mathcal{R}_1(N).
\]
Since 
\[
\mathcal{R}_1(\Phi_{t_j} ([\rho_j])) = \mathcal{R}_1 ([\rho_j])\in \mathcal{R}_1(N)
\]
by (\ref{R1action}), we know that 
\[
(f_j^1)_*(\mathcal{R}_1(N)) \cap \mathcal{R}_1(N) \ne \emptyset.
\]
Thus, $f_j ^1$ is homotopic and hence isotopic \cite[Theorem 6.4]{epstein1966} to the identity on $S_1$ by our choice of $N$. Therefore, each $f_j$ must be a power of the Dehn twist of $S$ along $\alpha$, since the restriction is isotopic to the identity map on $S_1$. 

Suppose now that $\alpha$ is separating. As in Section~\ref{fungpdecomsep}, write $S\setminus\nu(|\alpha|)=S_1\cup S_2$. For each $j$ and $i=1,2$, we may assume without loss of generality that $f_j(|\alpha|)$ is isotopic to $|\alpha|$.  We further claim that $f(S_i)$ is isotopic to $S_i$ for each $i=1,2$. For contradiction, suppose that $f(S_1)$ is isotopic to $S_2$. Consider the simple closed curves $\{|\alpha|=\alpha_1, \alpha_2, \cdots, \alpha_{3g+b-3}\}$ in Proposition~\ref{prop:length}. If $3g+b-3=1$, $S$ must be the sphere with four boundary components. In this case, $f(S_i)$ is isotopic to $S_i$ as $f$ preserves each boundary component. Hence, we may assume that $3g+b-3>1$. Let $\alpha_2\subset S_1$. Since $f^{-1}(\alpha_2)$ is not isotopic to $\alpha_2$, we know that 
\[
|l_{[\rho_2]}(\alpha_2)-l_{[\rho_1]}(f^{-1}(\alpha_2))|>0
\]
by our choice of $N$. However, we have 
\[
l_{[\rho_2]}(\alpha_2)=l_{f_*([\rho_1])}(\alpha_2)=l_{[\rho_1]}(f^{-1}(\alpha_2)),
\]
a contradiction.  Thus, $f(S_i)$ is isotopic to $S_i$.  Now we may isotope  $f_j$  so that $f_j(\nu(|\alpha|))=\nu(|\alpha|)$ for some tubular neighborhood $\nu(|\alpha|)$ and  $f_j(S_i) = S_i$, for each $i=1,2$. Let $f_j ^i:=f_j|S_i: S_i \ra S_i$ denote the restricted homeomorphisms.

Observe that  
\[
\mathcal{R}_1 ((f_j)_*(\Phi_{t_j} ([\rho_j]))) = (f_j ^1)_* (\mathcal{R}_1(\Phi_{t_j}([\rho_j]))) = \mathcal{R}_1([\rho_j'])\in \mathcal{R}_1(N).
\]
Since 
\[
\mathcal{R}_1(\Phi_{t_j} ([\rho_j])) = \mathcal{R}_1 ([\rho_j])\in \mathcal{R}_1(N)
\]
by (\ref{R1actionsep}), we know that 
\[
(f_j^1)_*(\mathcal{R}_1(N)) \cap \mathcal{R}_1(N) \ne \emptyset.
\]
Thus, $f_j ^1$ are homotopic and hence isotopic \cite[Theorem 6.4]{epstein1966} to the identity on $S_1$ by our choice of $N$. Similarly, we conclude that $f_j ^2$ are isotopic to the identity on $S_2$. This shows that each $f_j$ is a power of the Dehn twist of $S$ along $\alpha$.

Hence, either $\alpha$ is separating or non-separating, each $f_j$ must be a power of the Dehn twist of $S$ along $\alpha$.  It follows that  $(\phi_j)_*=\Phi^{m_jT,\alpha}$ for some $m_j\in \mathbb{Z}$. We refer to Section~\ref{sec:bending} for the definition of the function $T$.

Therefore, we have 
\[[\rho_j']=\Phi^{m_j T([\rho_j]),\alpha}(\Phi_{t_j}([\rho_j])). \]
Here we used the fact that $T(\Phi_{t_j}([\rho_j])) = T([\rho_j])$ as $\Phi_t$ preserves the conjugacy class of $\rho_j (\alpha_{x_0})$.

Recall that $\Phi_{t_j}= \Phi^{t_j Q_{\mathsf{G}},\alpha}$, where $Q_{\mathsf{G}}$ lies in the boundary of the positive Weyl chamber. Thus, we have 
$\Phi_{t_j}\Phi^{m_j T([\rho_j]),\alpha} = \Phi^{ Q'_j,\alpha}$ where
\[
Q'_j:=t_j Q_{\mathsf{G}}-m_j\lambda(\rho_j(\alpha)).
\]
We know that $\rho_j$ are $\Theta$-Anosov \cite[Theorem B]{guichard2021}. Therefore, all $\lambda(\rho_j(\alpha))$ stay in the interior of the positive Weyl chamber. Because $Q_{\mathsf{G}}$ lies on the boundary of the positive Weyl chamber, we know that  $\lambda(\rho_j(\alpha))$  are not constant multiples of $Q_{\mathsf{G}}$. Since $[\rho_j]$ are situated in the compact set $N$, there is a lower bound $0< \varphi_N< \pi$ on the angles between $Q_{\mathsf{G}}$ and $\lambda(\rho_j(\alpha))$. We estimate 
\[
\|Q_j '\|=||t_j Q_{\mathsf{G}} - m_j \lambda(\rho_j(\alpha))|| \geq |t_j|\cdot ||Q_{\mathsf{G}}|| \sin(\varphi_N)
\]
for all $j$ by projecting $Q_{\mathsf{G}}$ to the orthogonal hyperplane to $\lambda(\rho_j(\alpha))$. 

We now have that $\Phi^{Q'_j, \alpha}(N)\cap N \ne \emptyset$ for all $j$ and that $\|Q_j'\|\to \infty$ as $j\to \infty$ since $|t_j|\to \infty$ as $j\to \infty$. However, this is not possible by Proposition~\ref{prop:proper}. Hence, there exists $D_N$ with desired properties.

(ii) Suppose that $\Pi| \Phi_t(N)$ is not injective. Then there is $\phi\in \Mod(S)$ such that $\phi_*[\rho_1] =[\rho_2]$ for $[\rho_1], [\rho_2]\in \Phi_t(N)$.  As above, we choose  a representative $f$ of $\phi$ such that $f(|\alpha|) = |\alpha|$. As observed in the proof of (i),  $f$ preserves the orientation of $\alpha$.

Now, one may apply Lemma~\ref{lem:commute} to compute
\[ \phi_\ast(\Phi_{-t}([\rho_1])) = \Phi_{ -t}\phi_\ast([\rho_1]) = \Phi_{-t}([\rho_2]). \]
This implies that $\Pi$ is not injective on $N$ violating our choice of $N$. This proves (ii).

(iii) Let us choose $D_N$ as in (i). Then there are 
$[\rho] \in \Phi_t(N)$ and $[\rho']\in \Phi_s(N)$
where $f_\ast([\rho]) = [\rho']$ for a homeomorphism $f:S\to S$. Assume without loss of generality that $s> t$. 

As above,   $f$ can be isotoped so that  $f(|\alpha|)=|\alpha|$. Also we know that $f$ preserves the orientation of $\alpha$. By Lemma \ref{lem:commute}, we obtain 
\[ f_\ast( \Phi_{-t}([\rho]) )= \Phi_{- t}( f_\ast([\rho])) = \Phi_{- t}([\rho']) \in \Phi_{s- t}(N). \]
Since $\Phi_{-t}([\rho]) \in N$, and  $s-t > D_N$,  (iii) follows from (i). 
\end{proof}

\begin{proof}[Proof of Theorems~\ref{thm:main}, \ref{thm:maximal}] As the maximal case, Theorem~\ref{thm:maximal}, can be shown by the same argument, we only present the proof for the Hitchin case. 

    Let $S$ be a closed orientable surface with genus $g>1$. Let $N$ be chosen as in Proposition~\ref{prop:length} (or Proposition~\ref{prop:lengthmax} for the maximal case) and let $D_N$ be the constant from Lemma~\ref{lem:disj}. We choose a sequence $0<t_1<t_2<\cdots$ of reals so that $t_{i+1} - t_i > D_N$. Lemma~\ref{lem:disj} shows that $\Pi(\Phi_{t_i}(N))$ are open and pairwise  disjoint. Since $\Phi_t$ is $\vol_n$-preserving, $\Pi(\Phi_{t_1}(N)), \Pi(\Phi_{t_2}(N)),$ and so on have identical positive volume in $\mathcal{M}_\mathsf{G}(S)$. This proves  Theorems~\ref{thm:main} and \ref{thm:maximal}.
\end{proof}

\begin{proof}[Proof of Theorems~\ref{thm:main2}]
Now assume that $S$ is compact orientable with negative Euler characteristics.  If $S$ is not a sphere with three boundary components, we just apply the previous proof with the choice of open set $N$ from Proposition~\ref{prop:lengthrel}. 

    To complete the proof, assume that $S$ is a sphere with three boundary components. We need the following lemma which relates $\omega_{ABG}$ and $\omega_{GHJW}$.
\begin{lem}\label{ABGandGHJW}
    Suppose that a closed orientable surface $S$ of genus $>1$ is decomposed along a set of pairwise disjoint simple closed curves $\xi_1, \cdots, \xi_k$ into two compact subsurfaces $S_1$ and $S_2$ of negative Euler characteristics.  Let $\mathbf{B}$ be a choice of boundary holonomy for $S_1$. Let $I: \Hit_n(S) \to \Hit_n(S_1, \mathbf{B})$ be the natural map induced from the inclusion $\pi_1(S_1) \to \pi_1(S)$. Then for any vector fields $X,Y$ tangent to a submanifold of the form
    \[
    \{[\rho]\in \Hit_n(S)\,|\,[\rho|\pi_1(S_2)]=[\rho_2]\},
    \]
    for some chosen $[\rho_2]\in \Hit_n(S_2)$, we have $\omega_{ABG}(X,Y) = \omega_{GHJW}(dI( X), dI( Y))$.
\end{lem}
\begin{proof}
    This is a special case of Theorem~4.5.7 in \cite{CJK20}.
\end{proof}

When $S$ is a sphere with three boundary components, we observe that $\Mod(S)$ is trivial. Therefore, we have $\mathcal{M}_n(S,\mathbf{B}) = \Hit_n(S, \mathbf{B})$. 

Consider the double $\widehat{S}$ of $S$. We fix the pants-decomposition $\mathcal{P}$ that contains $S$ as one of components. By Lemma~8.6 of \cite{SWZ2020}, eruption flows and hexagonal flows on $\Hit_n(\widehat{S})$ with respect to $\mathcal{P}$ are tangent to the submanifold in Lemma~\ref{ABGandGHJW} and form a complete system of commuting Hamiltonian flows. Hence, by the Lemma~\ref{ABGandGHJW}, their Hamiltonian functions give rise to unbounded global Darboux coordinates for $(\Hit_n(S,\mathbf{B}),\omega_{GHJW})$. This completes the proof for Theorem~\ref{thm:main2}.
\end{proof}

\section{Questions}\label{sec:question}
Our technique has room for further generalizations. In this section, we discuss to what extent our proof can be applied.

\subsection{Generalizing Theorems~\ref{thm:main} and \ref{thm:main2}}
There are some cases leftover in Theorems~\ref{thm:main} and \ref{thm:main2}. We can perhaps propose the following bold question. Given a semi-simple Lie group $\mathsf{G}$ and a closed orientable surface $S$, a \emph{higher Teichm\"uller space} $\mathfrak{T}(S,\mathsf{G})$ is a connected component of the character variety $\operatorname{Hom}(\pi_1(S), \mathsf{G})/\mathsf{G}$ consisting solely of discrete and faithful representations. Hitchin components and the space of maximal representations are the most well-known examples. However, higher Teichm\"uller spaces do not always exist. In fact, it is recently shown \cite{bradlow2024,guichard2021} that the existence of so-called a $\Theta$-positive structure of $\mathsf{G}$ guarantees the existence of a higher Teichm\"uller space. 

Given a Lie group $\mathsf{G}$ with a $\Theta$-positive structure, the higher Teichm\"uller space $\mathfrak{T}(S,\mathsf{G})$ is precisely the space of $\Theta$-positive representations. Since these representations are $\Theta$-Anosov \cite{guichard2021}, one may consider the moduli space $\mathcal{M}_{\mathsf{G},\Theta}(S) = \mathfrak{T}(S,\mathsf{G}) / \Mod(S)$ and ask whether it has infinite total Atiyah-Bott-Goldman volume. 

\begin{question}\label{qu:higherteichmuller}
Suppose that the real rank of $\mathsf{G}$ is bigger than 1. Does $\mathcal{M}_{\mathsf{G},\Theta} (S)$ have infinite total Atiyah-Bott-Goldman volume?
\end{question}

Given the generality of Goldman flow, we speculate that the answer to Question~\ref{qu:higherteichmuller} should be yes.

\subsection{``Thick part'' of the Hitchin-Riemann moduli spaces} 
For a given positive real number $\epsilon>0$, define 
\[
    \operatorname{Hit}_{\mathsf{G}} ^\epsilon(S) = \{[\rho]\in \operatorname{Hit}_{\mathsf{G}}(S)\,|\,l_{\rho}(\gamma)<\epsilon\,\text{ for some essential simple closed curve }\gamma\}
\]
and
    \[
    \mathcal{M}_{\mathsf{G}} ^\epsilon(S) = \operatorname{Hit}_{\mathsf{G}} ^\epsilon (S) / \operatorname{Mod}(S).
    \]
We may call $ \mathcal{M}_{\mathsf{G}} ^\epsilon(S)$ the $\epsilon$\emph{-thin part} of the Hitchin-Riemann moduli space $\mathcal{M}_{\mathsf{G}}(S)$.

Our main theorem actually shows that, for  $\epsilon=D_\mathsf{G}$, the volume of $\mathcal{M}^\epsilon _{\mathsf{G}}(S)$ is infinite if the real rank of $\mathsf{G}>1$. One may ask whether the volumes of $\epsilon$-thick parts are infinite as well.
\begin{question} \label{thickpart}
        Suppose that the real rank of $\mathsf{G}$ is at least 2. For any $\epsilon>0$, does the $\epsilon$-thick part $\mathcal{M}_{\mathsf{G}}(S) \setminus \mathcal{M}_n^\epsilon(S) $  also have infinite volume?
\end{question}

A contrasting classical result is Mumford's compactness theorem, stating that $\mathcal{M}_2(S) \setminus \mathcal{M}_2^\epsilon(S) $  is compact for any $\epsilon$. 

Our present tool no longer works on $\epsilon$-thick parts. However, we are currently working on Question~\ref{thickpart} for the $n=3$ case by using a slightly improved technique and strongly believe that we proved this.

\bibliography{aff3bib0311.bib} 
\bibliographystyle{amsplain}

\end{document}